\documentclass[12pt]{amsart}
\usepackage[all]{xy}

\usepackage{graphicx}  

\usepackage{amsmath}
\usepackage{amssymb}
\usepackage{amsfonts}
\usepackage{mathrsfs}
\usepackage{mathtools}

\newcommand{\Cdb}{\ensuremath{\mathbb{C}}}
\newcommand{\Ddb}{\ensuremath{\mathbb{D}}}

\newcommand{\Pdb}{\ensuremath{\mathbb{P}}}

\newcommand{\Rdb}{\ensuremath{\mathbb{R}}}

\newcommand{\Zdb}{\ensuremath{\mathbb{Z}}}

\newcommand{\M}{\mbox{${\mathcal M}$}}

\renewcommand{\O}{\mbox{${\mathcal O}$}}
\renewcommand{\P}{\mbox{${\mathcal P}$}}

\newcommand{\R}{\mbox{${\mathcal R}$}}
\renewcommand{\S}{\mbox{${\mathcal S}$}}

\newcommand{\norm}[1]{\Vert#1\Vert}
\newcommand{\bignorm}[1]{\bigl\Vert#1\bigr\Vert}
\newcommand{\Bignorm}[1]{\Bigl\Vert#1\Bigr\Vert}
\newcommand{\biggnorm}[1]{\biggl\Vert#1\biggl\Vert}

\newtheorem{theorem}{Theorem}[section]
\newtheorem{lemma}[theorem]{Lemma}
\newtheorem{corollary}[theorem]{Corollary}
\newtheorem{proposition}[theorem]{Proposition}
\newtheorem{definition}[theorem]{Definition}
\theoremstyle{remark}
\newtheorem{remark}[theorem]{\bf Remark}
\theoremstyle{definition}

\numberwithin{equation}{section}

\begin{document}

\title[]{$H^\infty$ functional calculus and square function estimates for Ritt operators}

\author{Christian Le Merdy}
\address{Laboratoire de Math\'ematiques\\ Universit\'e de  Franche-Comt\'e
\\ 25030 Besan\c con Cedex\\ France}
\email{clemerdy@univ-fcomte.fr}

\date{\today}

\thanks{The author is supported by the research program ANR 2011 BS01 008 01}

\begin{abstract}
A Ritt operator $T\colon X\to X$ on Banach space is a power bounded
operator satisfying an estimate $n\norm{T^{n} -T^{n-1}}\leq C\,$.
When $X=L^p(\Omega)$ for some $1<p<\infty$, we study the
validity of square functions estimates $\bignorm{\bigl(\sum_k
k\vert T^{k}(x) - T^{k-1}(x)\vert^2\bigr)^{\frac{1}{2}}}_{L^p}\lesssim\norm{x}_{L^p}$
for such operators. We show that $T$ and $T^*$ both satisfy such
estimates if and only if $T$ admits a bounded functional calculus with respect
to a Stolz domain. This is a single operator analog of the famous
Cowling-Doust-McIntosh-Yagi characterization of bounded $H^\infty$-calculus
on $L^p$-spaces by the boundedness of certain Littlewood-Paley-Stein square
functions. We also prove a similar result on Hilbert space. Then
we extend the above to more general Banach spaces,
where square functions have to be defined in terms of certain
Rademacher averages. We focus on noncommutative $L^p$-spaces, where
square functions are quite explicit, and 
we give applications, examples and illustrations on those
spaces, as well as on classical $L^p$.
\end{abstract}

\maketitle


\bigskip

\section{Introduction}
Let $X$ be a Banach space and let $T\colon X\to X$ be a bounded operator.
If $F\subset\Cdb$ is any compact set containing the spectrum of $T$, a natural 
question is whether there is an estimate 
\begin{equation}\label{1Spectral}
\norm{\varphi(T)}\leq 
K\sup\bigl\{\vert \varphi(\lambda))\vert\, :\, \lambda\in F\bigr\}
\end{equation}
satisfied by all rational functions $\varphi$. 
The mapping $\varphi\mapsto \varphi(T)$ on rational functions  
is the most elementary form of a `holomorphic functional calculus' 
associated to $T$ and (\ref{1Spectral}) means that 
this functional calculus is bounded in an appropriate sense.

The most famous such functional calculus estimate
is von Neumann's inequality, which says that if  $F=\overline{\Ddb}$
is the closed unit disc centered at $0$, then (\ref{1Spectral}) holds true
with $K=1$ for any contraction $T$ on Hilbert space. 
Von Neumann's inequality was a source of inspiration for the development 
of various topics around functional calculus estimates on Hilbert space,
including polynomial boundedness, $K$-spectral sets and related similarity problems.
We refer the reader to \cite{BBC, Pa1, Pa2, P4} and the references therein for a large information.
See also \cite{DD, C} for striking results in the case when 
$F$ is equal to the numerical range of $T$.

When $X$ is a non Hilbertian Banach space, our knowledge on operators $T\colon X\to X$
and compact sets $F$ satisfying (\ref{1Spectral})
for some $K\geq 1$ is quite limited. Positive examples are provided by
scalar type operators (see \cite{Dow}).
A more significant observation is that this issue
is closely related to $H^\infty$-functional
calculus associated to sectorial operators and indeed, that
topic will play a key role in this paper. $H^\infty$-functional
calculus was introduced by McIntosh and his co-authors in \cite{CDMY, MI} and was then
developed and applied successfully to various areas, in particular to the study of
maximal regularity for certain PDE's, to harmonic analysis of semigroups,
and to multiplier theory. We refer the reader to \cite{KW} for relevant information.

\bigskip
In this paper we deal with holomorphic functional calculus for Ritt operators. 
Recall that by definition, $T\colon X\to X$ is a Ritt operator provided that $T$ is
power bounded and there exists a constant $C>0$ such that $n\norm{T^{n} -T^{n-1}}\leq C\,$
for any integer $n\geq 1$. In this case, the spectrum of $T$ is included in the closure
$\overline{B_\gamma}$ of a Stolz domain of the unit disc, see Section 2 and 
Figure 1 below for details. In accordance with the preceding discussion,
this leads to the question whether $T$ satisfies an estimate (\ref{1Spectral}) for 
$F=\overline{B_\gamma}$. We will say that $T$ has a bounded 
$H^\infty(B_\gamma)$ functional calculus in this case (this terminology
will be justified in Section 2). The general problem motivating the present work
is to characterize Ritt operators having a bounded 
$H^\infty(B_\gamma)$ functional calculus for some $\gamma\in\bigl(0,\frac{\pi}{2}\bigr)$ and
to exhibit explicit classes of operators satisfying this property.

If $X=H$ is a Hilbert space and $T\colon H\to H$ is a bounded operator, we 
define the `square function'
\begin{equation}\label{1SFH}
\norm{x}_{T}\,=\,\biggl(\sum_{k=1}^{\infty} k\bignorm{T^{k}(x) - T^{k-1}(x)}_H^2\biggr)^{\frac{1}{2}},
\qquad x\in H.
\end{equation}
Likewise for any measure space $(\Omega,\mu)$, for any $1\leq p<\infty$ and for any 
$T\colon L^p(\Omega)\to L^p(\Omega)$, we consider 
\begin{equation}\label{1SFLp}
\norm{x}_{T}\,=\, \biggnorm{\biggl(\sum_{k=1}^{\infty} k\bigl 
\vert T^{k}(x) - T^{k-1}(x)\bigr\vert^2\biggr)^{\frac{1}{2}}}_{L^p(\Omega)},\qquad x\in L^p(\Omega).
\end{equation}

Let $T\colon X\to X$ be a Ritt operator on either $X=H$ or $X=L^p(\Omega)$.
It was implicitely proved in \cite{LMX1} that if $T$ has a bounded $H^{\infty}(B_\gamma)$ functional
calculus for some $\gamma\in \bigl(0,\frac{\pi}{2}\bigr)$, then 
it satisfies a uniform estimate $\norm{x}_{T}\lesssim\norm{x}$.

This  paper has two main purposes. 
First we establish a converse to this 
result and prove the following. (Here 
$p'=\frac{p}{p-1}$ is the conjugate number of $p$.)

\begin{theorem}\label{1MainLp}
Let $T\colon L^p(\Omega)\to L^p(\Omega)$ be a Ritt operator, with $1<p<\infty$. 
The following assertions are equivalent.
\begin{enumerate}
\item [(i)] The operator $T$ admits a bounded $H^\infty(B_\gamma)$ functional 
calculus for some $\gamma\in\bigl(0,\frac{\pi}{2}\bigr)$.
\item [(ii)] The operator $T$ and its adjoint $T^*\colon L^{p'}(\Omega)\to L^{p'}(\Omega)$
both  satisfy uniform estimates
$$
\norm{x}_{T}\,\lesssim\,\norm{x}_{L^p}\qquad\hbox{and}\qquad
\norm{y}_{T^*}\,\lesssim\,\norm{y}_{L^{p'}}
$$
for $x\in L^p(\Omega)$ and $y\in L^{p'}(\Omega)$.
\end{enumerate}
\end{theorem}

We  also prove a similar result for Ritt operators on Hilbert space.

Second, we investigate relationships between the existence of 
a bounded  $H^{\infty}(B_\gamma)$ functional calculus and adapted  square
function estimates on general Banach spaces. We pay a special attention
to noncommutative $L^p$-spaces and prove square function estimates 
for large classes of Schur multipliers and selfadjoint Markov operators
on those spaces.

Ritt operators can be considered as discrete analogs of sectorial
operators of type $<\frac{\pi}{2}$, as explained e.g. in \cite{Bl1,Bl2}
or \cite[Section 2]{LMX1}. According to this analogy, 
Theorem \ref{1MainLp} and its Hilbertian
counterpart should be regarded as discrete analogs of the 
main results of \cite{CDMY, MI} showing the equivalence 
between the boundedness of $H^\infty$-functional calculus and 
some square function estimates for sectorial operators. 
Likewise, in the noncommutative setting,
our results are both an analog and an extension of 
the main results of the memoir
\cite{JLX}.

The definitions of the 
discrete square functions (\ref{1SFH}) and (\ref{1SFLp}) go back at
least to \cite{S2}, where they were used to study selfadjoint
Markov operators and diffusion semigroups on classical (=commutative)
$L^p$-spaces. They appeared in
the context of Ritt operators in \cite{KP} and \cite{LMX1,LMX2}.

\bigskip
We now turn to a brief description of the paper. 
In Sections 2 and 3, we introduce $H^\infty(B_\gamma)$ functional 
calculus and square functions for Ritt operators, and we prove basic 
preliminary results. Our definition of square functions on general
Banach spaces relies on Rademacher averages. Regarding such averages as
abstract square functions is a well-known principle, see e.g. 
\cite{PX, KaW2, JLX} for illustrations. If $T$ is a Ritt operator, then
$A=I_X - T$ is a sectorial operator and we show in Section 4 that 
$T$ has a bounded $H^\infty(B_\gamma)$ functional calculus for 
some $\gamma<\frac{\pi}{2}$ if and only if $A$ 
has a bounded $H^\infty(\Sigma_\theta)$ functional calculus for 
some $\theta<\frac{\pi}{2}$. This observation, stated as Proposition \ref{4TA},
provides a tool
to transfer bounded $H^\infty$-calculus results from 
the sectorial setting to Ritt operators. There is apparently no similar way to compare square
functions associated to $T$ to square functions associated to
$A$. This is at the root of most of the difficulties in our analysis of Ritt
operators. Proposition \ref{4TA} will be applied
in Section 8, where we give applications and illustrations on
Hilbert spaces, classical $L^p$-spaces, and noncommutative
$L^p$-spaces.

We will make use of $R$-boundedness and the notion of $R$-Ritt
operators. That class was introduced by Blunck \cite{Bl1, Bl2} as
a discrete counterpart of $R$-sectorial operators. Our first main result,
proved in Section 5, says that if a Ritt operator
$T\colon L^p(\Omega)\to L^p(\Omega)$ satisfies condition (ii) in Theorem 
\ref{1MainLp} above, then it is actually an $R$-Ritt operator.
In Section 6, we show that on Banach spaces $X$ with finite cotype, 
any Ritt operator $T\colon X\to X$ with a bounded
$H^\infty(B_\gamma)$ functional calculus satisfies square function 
estimates. This is based on the study of a strong form of $H^\infty$-functional calculus
called `quadratic $H^\infty$-functional calculus', where scalar valued holomorphic functions
are replaced by $\ell^2$-valued ones. Section 7 is devoted to the 
converse problem of whether square function estimates for
$T$ and $T^*$ imply a bounded
$H^\infty(B_\gamma)$ functional calculus. We show that this holds
true whenever $T$ is
$R$-Ritt, and complete the proofs of Theorem \ref{1MainLp} and
similar equivalence results.

\bigskip

We finally give a few notation to be used along this paper. We let $B(X)$ denote
the algebra of all bounded operators on $X$ and we let 
$I_X$ denote the identity operator on
$X$ (or simply $I$ if there is no ambiguity on $X$). 
We let $\sigma(T)$ denote the spectrum of an operator $T$ (bounded or not)
and we let $R(\lambda,T)=(\lambda I_X -T)^{-1}$ denote
the resolvent operator when $\lambda$ belongs to the resolvent set
$\Cdb\setminus\sigma(T)$. Next, we let ${\rm Ran}(T)$ and ${\rm Ker}(T)$ denote the
range and the kernel of $T$, respectively.

For any $a\in\Cdb$ and $r>0$, we let $D(a,r)$ denote the open disc or radius $r$ centered
at $a$. Also, we let $\Ddb=D(0,1)$ denote the open unit disc. For any non empty open
set $\O\subset\Cdb$ and any Banach space $Z$, we let $H^{\infty}(\O;Z)$ denote the
space of all bounded holomorphic functions $\varphi\colon\O\to Z$. This is a Banach space
for the supremum norm
$$
\norm{\varphi}_{H^{\infty}(\footnotesize{\O};Z)}
\,=\,
\sup\bigl\{\norm{\varphi(\lambda)}_Z\, :\, \lambda\in \O\bigr\}.
$$
In the scalar case, we write $H^{\infty}(\O)$ instead of $H^{\infty}(\O;\Cdb)$
and $\norm{\varphi}_{\infty,\footnotesize{\O}}$ instead of 
$\norm{\varphi}_{H^{\infty}(\footnotesize{\O})}$. Finally we let $\P$ denote the
algebra of complex polynomials.

In Theorem \ref{1MainLp} and later on in the paper we use 
the notation $\lesssim$ to indicate an inequality
up to a constant which does not depend on the particular element to which it applies. Then
$A(x) \approx B(x)$ will mean that we both have $A(x)\lesssim B(x)$ and $B(x)\lesssim A(x)$.

\medskip
\section{Ritt operators and their functional calculus}
We start this section
with some classical background on the $H^\infty$-functional calculus associated to
sectorial operators. The construction and basic properties below go back to 
\cite{CDMY,MI}, see also \cite{KaW1,LM1} for complements.

For any $\omega\in (0,\pi)$, we let 
\begin{equation}\label{2Sector}
\Sigma_\omega=\bigl\{z\in\Cdb^*\, :\, \bigl\vert{\rm Arg}(z)\bigr\vert <\omega\bigr\}
\end{equation}
be the open sector of angle $2\omega$ around the positive real axis $(0,\infty)$.

Let $X$ be a Banach space. 
We say that a closed linear operator $A\colon D(A)\to X$ with dense domain $D(A)\subset X$ is sectorial
of type $\omega$ if $\sigma(A)\subset\overline{\Sigma_\omega}$ 
and for any $\nu\in(\omega,\pi)$, the set 
\begin{equation}\label{2Sectorial}
\{zR(z,A)\, :\, z\in\Cdb\setminus \overline{\Sigma_\nu}\}
\end{equation}
is bounded. 

For any $\theta\in (0,\pi)$, let $H^\infty_0(\Sigma_\theta)$ denote the 
algebra of all bounded
holomorphic functions $f\colon\Sigma_\theta\to\Cdb$  for which 
there exists two positive real numbers $s,c>0$ such that
$$
\vert f(z)\vert\leq\, c\,\frac{\vert z\vert^s}{1+\vert z\vert^{2s}}\,,
\qquad z\in\Sigma_\theta.
$$

Let $0<\omega<\theta<\pi$ and let $f\in H^\infty_0(\Sigma_\theta)$. Then we set
\begin{equation}\label{2CauchySec}
f(A)\,=\,\frac{1}{2\pi i}\,\int_{\partial\Sigma_\nu} f(z) R(z,A)\, dz\,,
\end{equation}
where $\nu\in(\omega,\theta)$ and the boundary $\partial\Sigma_\nu$ 
is oriented counterclockwise. The sectoriality condition
ensures that this integral is absolutely convergent and defines an element
of $B(X)$. Moreover by Cauchy's Theorem, this definition does not
depend on the choice of $\nu$. Further the resulting mapping $f\mapsto 
f(A)$ is an algebra homomorphism from $H^\infty_0(\Sigma_\theta)$ into $B(X)$
which is consistent with the usual functional calculus for rational functions.

We say that $A$ admits a bounded $H^\infty(\Sigma_\theta)$ functional calculus
if the latter homomorphism is bounded, that is, there exists a constant $K>0$
such that 
$$
\norm{f(A)}\leq K\norm{f}_{\infty, \Sigma_\theta},\qquad f\in
H^\infty_0(\Sigma_\theta).
$$
If $A$ has a dense range and admits a bounded $H^\infty(\Sigma_\theta)$ 
functional calculus, then the above homomorphism naturally extends to 
a bounded homomorphism $f\mapsto f(A)$ from the whole space
$H^\infty(\Sigma_\theta)$ into $B(X)$.

\bigskip
It is well-known that the above construction can be adapted to various
contexts, see e.g. \cite{Ha} and \cite{FMI}. We shall briefly explain below such a
functional calculus construction 
for Ritt operators. We first recall some background on this class.

We say that an operator $T\colon X\to X$ is a Ritt operator
provided that the two sets 
\begin{equation}\label{2Sets}
\bigl\{T^n\, :\, n\geq 0\bigr\}
\qquad\hbox{and}\qquad
\bigl\{n(T^n-T^{n-1})\, :\, n\geq 1\bigr\}
\end{equation}
are bounded. The following spectral characterization is crucial: $T$ is a Ritt
operator if and only if 
$$
\sigma(T)\subset\overline{\Ddb}
\qquad\hbox{and}\qquad
\bigl\{(\lambda-1)R(\lambda,T)\, :
\, \vert\lambda\vert>1\bigr\}\quad\hbox{is bounded}.
$$
Indeed this condition is often taken as a definition for Ritt
opertors. We refer to \cite{Ly,NZ} for this 
characterization and also to 
\cite{N}, which contains the key argument, and to \cite{Bl1,Bl2}
and \cite[Section 2]{LMX1} for complements. Let 
$$
A=I-T.
$$
It follows from the above referred papers that $T$ is a Ritt operator
if and only if 
\begin{equation}\label{2Ritt}
\sigma(T)\subset \Ddb\cup\{1\}
\qquad\hbox{and}\qquad
A \ \hbox{is a sectorial operator of type}\,<   \tfrac{\pi}{2}\,.
\end{equation}
We will need quantitative versions of the above equivalence property. 
For that purpose we introduce the Stolz domains $B_\gamma$ as on Figure 1 below. 
Namely for any angle $\gamma\in\bigl(0,\frac{\pi}{2}\bigr)$, we let 
$B_\gamma$ be the interior of the convex hull of $1$ and the 
disc $D(0,\sin\gamma)$.

\begin{figure}[ht]
\vspace*{2ex}
\begin{center}
\includegraphics[scale=0.4]{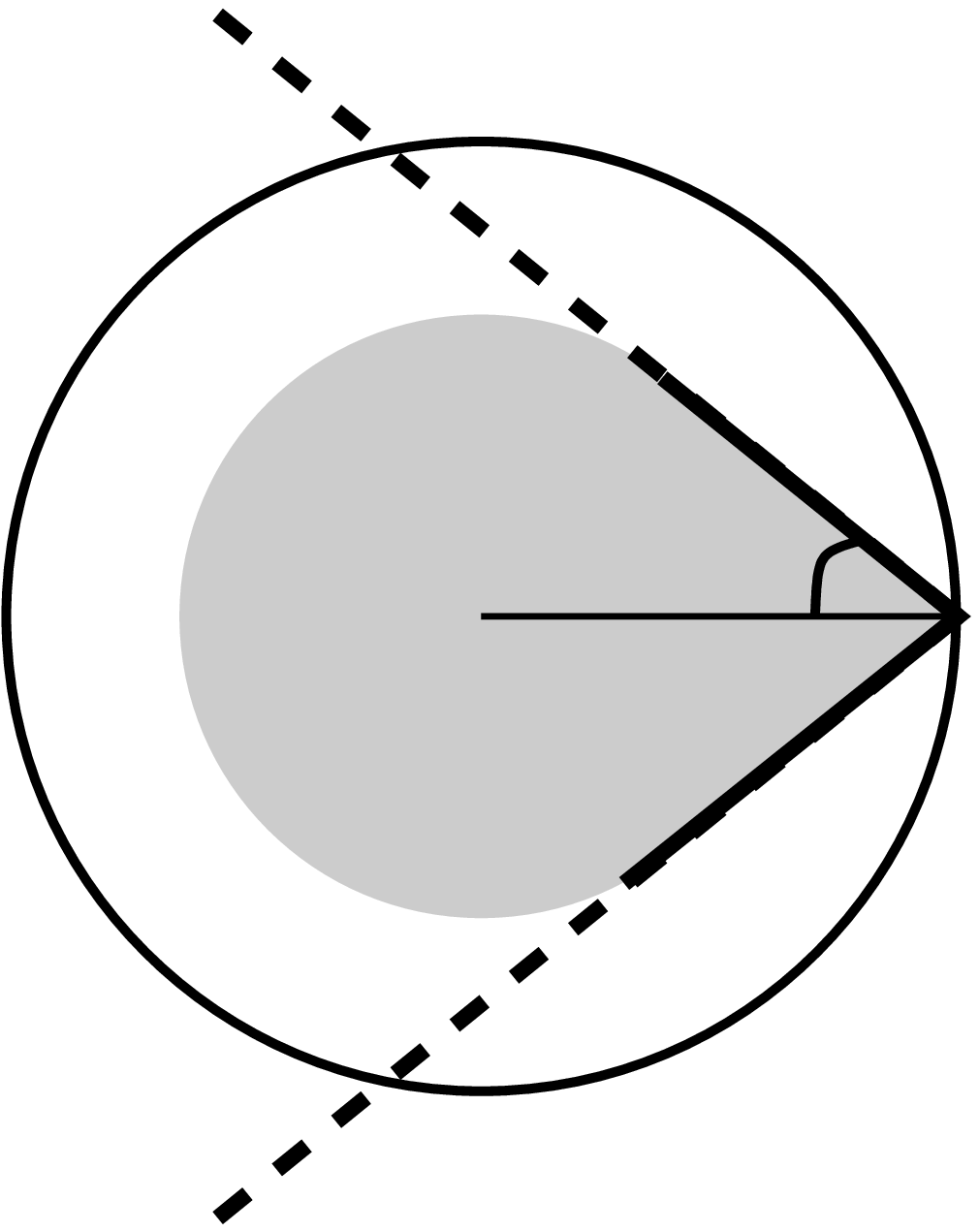}
\begin{picture}(0,0)
\put(-2,65){{\footnotesize $1$}}
\put(-68,65){{\footnotesize $0$}}
\put(-32,81){{\footnotesize $\gamma$}}
\put(-55,85){{\small $B_\gamma$}}
\end{picture}
\end{center}
\caption{\label{f1}}
\end{figure}

\begin{lemma}\label{2Blunck} An operator $T\colon X\to X$ is a Ritt operator if
and only if there exists an angle $\alpha\in \bigl(0,\frac{\pi}{2}\bigr)$ such that
\begin{equation}\label{2Blunck1}
\sigma(T)\subset \overline{B_\alpha}
\end{equation}
and for any $\beta\in\bigl(\alpha,\frac{\pi}{2}\bigr)$, the set
\begin{equation}\label{2Blunck2}
\bigl\{(\lambda-1)R(\lambda,T)\, :\, \lambda\in\Cdb\setminus 
\overline{B_\beta} \bigr\}
\end{equation}
is bounded.

In this case, $A=I-T$ is a sectorial operator of type $\alpha$.
\end{lemma}

\begin{proof}
Assume that $T$ is a Ritt operator and let us apply (\ref{2Ritt}). Let $\omega\in\bigl(0,\frac{\pi}{2}\bigr)$
be a sectorial type of $A$. Then $\sigma(T)$ is both included in $\Ddb\cup\{1\}$ and in 
the cone $1-\overline{\Sigma_\omega}$ hence there exists $\omega\leq\alpha<\frac{\pi}{2}$ such that
$\sigma(T)\subset \overline{B_\alpha}$.

Consider the function $h$ on $\Cdb\setminus\sigma(T)$ defined by
$h(\lambda)=(\lambda-1)R(\lambda,T)$. This function is bounded on $\Cdb\setminus\overline{D}(0,2)$. Indeed
if we let $C_0=\sup_{n\geq 0}\norm{T^{n}}$, then writing
$$
R(\lambda,T)=\,\sum_{n=0}^{\infty}\frac{T^n}{\lambda^{n+1}}
$$
when $\vert\lambda\vert>1$, we have
$\norm{R(\lambda,T)}\leq C_0/(\vert\lambda\vert -1)$, and hence
$$
\vert\lambda -1\vert \norm{R(\lambda,T)}\,\leq C_0\,\frac{\vert\lambda\vert +1}{\vert\lambda\vert -1},
\qquad \lambda\in \Cdb\setminus\overline{\Ddb}.
$$

Let $\beta\in\bigl(\alpha,\frac{\pi}{2}\bigr)$. The compact set
\begin{equation}\label{2Lambda}
\Lambda_\beta=\,\bigl\{\lambda\in 1-\overline{\Sigma_\beta}\, :\, {\rm Re}(\lambda)\leq \sin^2\beta\quad
\hbox{and}\quad \sin\beta\leq\vert\lambda\vert\leq 2\bigr\}
\end{equation}
is included in the resolvent set of $T$, hence $h$ is bounded on $\Lambda_\beta$.
Furthermore 
$$
h(\lambda) = (1-\lambda)R\bigl((1-\lambda),A\bigr)
$$
and $A$ is sectorial of type $\alpha$. Consequently $h$ is bounded outside $1-\overline{\Sigma_\beta}$.
Altogether, this shows that $h$ is bounded outside $\overline{B_\beta}$.

The rest of the statement is obvious.
\end{proof}

The above lemma leads to the following.
\begin{definition}\label{2Type}
We say that $T\colon X\to X$ is a Ritt operator  
of type $\alpha\in\bigl(0,\frac{\pi}{2}\bigr)$ if it satisfies the conclusions
of Lemma \ref{2Blunck}. 
\end{definition}

Then we construct an $H^\infty$-functional calculus as follows.
For any $\gamma\in\bigl(0,\frac{\pi}{2}\bigr)$, we let $H^{\infty}_{0}(B_\gamma)
\subset H^{\infty}(B_\gamma)$
be the space of all bounded holomorphic functions $\varphi\colon 
B_\gamma\to \Cdb$ for which 
there exists two positive real numbers $s,c>0$ such that
\begin{equation}\label{2H0}
\vert\varphi(\lambda)\vert\leq c\vert 1-\lambda\vert^s,\qquad \lambda \in B_\gamma.
\end{equation}
Assume that $T$ has type $\alpha$ and $\gamma\in
\bigl(\alpha,\frac{\pi}{2}\bigr)$. Then for any $\varphi\in 
H^{\infty}_{0}(B_\gamma)$, we define
\begin{equation}\label{2Cauchy}
\varphi(T)\,=\,\frac{1}{2\pi i}\,\int_{\partial B_\beta}\varphi(\lambda) R(\lambda,T)\,d\lambda\,,
\end{equation}
where $\beta\in (\alpha,\gamma)$ and the boundary $\partial B_\beta$ is oriented counterclockwise.
The boundedness of $\bigl\{(\lambda-1)R(\lambda,T)\, :\, \lambda\in\partial B_\beta\setminus\{1\}
\bigr\}$ and 
the assumption (\ref{2H0}) imply that this integral is absolutely convergent and defines
an element of $B(X)$. It does not depend on $\beta$ and the mapping
$$
H^{\infty}_{0}(B_\gamma)\longrightarrow B(X),\qquad \varphi\mapsto \varphi(T),
$$
is an algebra homomorphism. Proofs of these facts are similar to the ones in the sectorial case.

We state a technical observation for further use.

\begin{lemma}\label{2rT}
Let $T$ be a Ritt operator of type $\alpha$. Then $rT$ is a Ritt operator for any $r\in(0,1)$ and:
\begin{itemize}
\item [(1)] 
For any $\beta\in\bigl(\alpha,\frac{\pi}{2}\bigr)$,
the set
$$
\bigl\{(\lambda-1)R(\lambda,rT)\, :\, r\in (0,1),\ \lambda\in\Cdb\setminus 
B_\beta\bigr\}
$$
is bounded;
\item [(2)]
For any $\gamma\in\bigl(\alpha,\frac{\pi}{2}\bigr)$ and any $\varphi\in H^{\infty}_{0}(B_\gamma)$,
$\varphi(T)=\lim_{r\to 1^{-}}\varphi(rT)$.
\end{itemize}
\end{lemma}

\begin{proof}
Consider $\beta\in\bigl(\alpha,\frac{\pi}{2}\bigr)$.
It is clear that for any $\lambda\in\Cdb\setminus B_\beta$ and any $r\in (0,1)$,
$\frac{\lambda}{r}\in\Cdb\setminus \overline{B_\beta}$, 
$\lambda\notin\sigma(rT)$ and we have
$$
(\lambda-1)R(\lambda,rT) =\frac{\lambda-1}{\lambda-r}\Bigl(\frac{\lambda}{r}-1\Bigr) R\Bigl(\frac{\lambda}{r},T\Bigr)
$$
Since the sets 
$$
\bigl\{(\lambda-1)(\lambda-r)^{-1}\, :\, r\in (0,1),\ \lambda\in\Cdb\setminus 
B_\beta, \bigr\}
$$
and 
$$
\bigl\{(\mu -1)R(\mu,T) \,:\, \mu\in\Cdb\setminus \overline{B_\beta} \bigr\}
$$ 
are bounded, we obtain (1).
 
Applying Lebesgue's Theorem to (\ref{2Cauchy}), the assertion (2) follows at once.
\end{proof}

Let $H^{\infty}_{0,1}(B_\gamma)\subset H^{\infty}(B_\gamma)$ be the linear span
of $H^{\infty}_{0}(B_\gamma)$ and constant functions. For any $\varphi = c +\psi$,
with $c\in\Cdb$ and $\psi\in H^{\infty}_{0}(B_\gamma)$, set $\varphi(T)= cI_X +\psi (T)$.
Then $H^{\infty}_{0,1}(B_\gamma)\subset H^{\infty}(B_\gamma)$ is a unital algebra
and $\varphi\mapsto \varphi(T)$ is a unital homomorphism from
$H^{\infty}_{0,1}(B_\gamma)$ into $B(X)$. Note that $H^{\infty}_{0,1}(B_\gamma)$
contains rational functions with poles off $\overline{B_\gamma}$, 
and hence polynomials.

For any $T$ as above
and any $r\in (0,1)$, $\sigma(rT) = r\sigma(T)\subset B_\beta$,
hence the definition of $\varphi(rT)$ provided by (\ref{2Cauchy}) is given by the usual 
Dunford-Riesz functional calculus of $rT$. It therefore follows from classical properties
of that functional calculus and the approximation Lemma \ref{2rT} that
for any rational function $\varphi$ with poles off $\overline{B_\gamma}$,
the above definition of $\varphi(T)$ coincides with 
the one obtained by substituing $T$ to the complex variable. This applies
in particular to any $\varphi\in\P$.

Likewise, recall that since $I-T$ is sectorial one can define
its fractional powers $(I-T)^\delta$ for any $\delta>0$. Then 
this bounded operator coincides with $\varphi_\delta(T)$,
where $\varphi_\delta$ is the element of $H^{\infty}_{0}(B_\gamma)$
given by $\varphi_\delta(\lambda) = (1-\lambda)^\delta$. See 
\cite[Section 6]{MI} and \cite[Chapter 3]{Ha} for similar results.

\begin{definition} Let $T$ be a Ritt operator of type $\alpha$
and let $\gamma\in\bigl(\alpha,\frac{\pi}{2}\bigr)$.
We say that $T$ admits a bounded $H^\infty(B_\gamma)$ functional calculus 
if there exists a constant $K>0$ such that 
$$
\norm{\varphi(T)}\leq K\norm{\varphi}_{\infty, B_\gamma},\qquad\varphi\in 
H^{\infty}_{0}(B_\gamma).
$$
\end{definition}

In this case $\varphi\mapsto \varphi(T)$ is a bounded homomorphism on $H^{\infty}_{0,1}(B_\gamma)$.
The next statement shows that the above functional calculus property 
can be tested on polynomials only.

\begin{proposition}\label{2Approx} 
A Ritt operator $T$ has 
a bounded $H^\infty(B_\gamma)$ functional calculus if and only if
there exists a constant $K\geq 1$ such that 
\begin{equation}\label{2Pol}
\norm{\varphi(T)}\leq K\norm{\varphi}_{\infty, B_\gamma}
\end{equation}
for any $\varphi\in\P$.
\end{proposition}

\begin{proof} The `only if' part is clear from the above discussion. 
To prove the `if' part, assume
(\ref{2Pol}) on $\P$ and consider
$\varphi\in H^{\infty}_{0}(B_\gamma)$. Let $r\in (0,1)$, let 
$r'\in(r,1)$ be an auxiliary real number and let $\Gamma$ be the boundary
of $r' B_\gamma$ oriented counterclockwise.

By Runge's Theorem (see e.g. \cite[Thm 13.9]{Rud}), there exists
a sequence $(\varphi_m)_{m\geq 1}$ of polynomials such that 
$\varphi_m\to \varphi$ uniformly on the compact set $r'\overline{B_\gamma}$.
Since $\sigma(rT)\subset r'B_\gamma$, we deduce that
$$
\varphi_m(rT)=\,\frac{1}{2\pi i}\int_{\Gamma} \varphi_m(\lambda)
R(\lambda, rT)\, d\lambda\ \longrightarrow\,
\frac{1}{2\pi i}\int_{\Gamma} \varphi(\lambda)
R(\lambda, rT)\, d\lambda\,=\varphi(rT),
$$
when $m\to \infty$. By (\ref{2Pol}),
$$
\norm{\varphi_m(rT)}\leq K \norm{\varphi_m}_{\infty,  rB_\gamma}\leq K
\norm{\varphi_m}_{\infty, r'B_\gamma}.
$$
Passing to the limit yields
$$
\norm{\varphi(rT)}\leq K\norm{\varphi}_{\infty, r'B_\gamma}.
$$
Finally letting $r\to 1$ and applying Lemma \ref{2rT}, (2), we deduce 
$\norm{\varphi(T)}\leq K\norm{\varphi}_{\infty, B_\gamma}$. 
\end{proof}

The above result is closely related to the following classical notion.

\begin{definition}\label{2Pb} We say that a bounded operator $T\colon X\to X$ is polynomially bounded
if there is a constant $K\geq 1$ such that
$$
\norm{\varphi(T)}\leq K\norm{\varphi}_{\infty,\Ddb},
\qquad \varphi\in\P.
$$
\end{definition}

Obviously any Ritt operator with a bounded $H^\infty(B_\gamma)$ functional 
calculus is polynomially bounded. See Proposition \ref{6Reduce} below for a partial 
converse. 

According to \cite[Prop. 5.2]{LM2}, there exist Ritt operators on Hilbert space which are not 
polynomially bounded. Thus there exist Ritt operators without any bounded 
$H^{\infty}(B_\gamma)$ functional 
calculus. Note that various such (counter-)examples can be derived from Proposition \ref{4TA} below
or from our Section 8.a.

\begin{remark} Let $T$ be a Ritt operator of type $\alpha$, let $\gamma\in\bigl(\alpha,
\frac{\pi}{2}\bigr)$, and assume that $I-T$ has a dense range. Then $I-T$ is 1-1 
by \cite[Thm. 3.8]{CDMY} and arguing as in \cite{CDMY,MI}, one can 
extend the definition of $\varphi(T)$ to any $\varphi\in H^{\infty}(B_\gamma)$.
Namely let $\psi(z)=1-z$ and for any $\varphi\in H^{\infty}(B_\gamma)$, set
$\varphi(T)= (I-T)^{-1}(\varphi\psi)(T)$, where $(\varphi\psi)(T)$ is defined by
(\ref{2Cauchy}) and $\varphi(T)$ is defined on $D\bigl(\varphi(T)\bigr)=\{x\in X\, :\,
(\varphi\psi)(T)x\in {\rm Ran}(I-T)\}$. It is easy to check that the domain
of $\varphi(T)$ contains ${\rm Ran}(I-T)$, so that $\varphi(T)$ in densely defined,
and that $\varphi(T)$ is closed. Consequently, $\varphi(T)$ is bounded if and only if
$D\bigl(\varphi(T)\bigr)=X$.

Assume that $T$ has a bounded $H^{\infty}(B_\gamma)$ functional calculus. Then 
$\varphi(T)$ is bounded for any $\varphi\in H^{\infty}(B_\gamma)$. Indeed let
$\psi_n(z) = (1-z)((1-z) +n^{-1})^{-1}$ for any integer $n\geq 1$. 
The sectoriality of $(I-T)$ ensures that $(\psi_n(T))_{n\geq 1}$
is bounded. Hence there is  a constant $K>0$ such that
$\norm{(\varphi\psi_n)(T)}\leq K\norm{\varphi}_{\infty, B_\gamma}$ for any $n\geq 1$.
It is easy to check that $(\varphi\psi_n)(T)x\to \varphi(T)x$
for any $x\in{\rm Ran}(I-T)$. This shows the boundedness of 
$\varphi(T)$, with the estimate
$\norm{\varphi(T)}\leq K\norm{\varphi}_{\infty, B_\gamma}$.
\end{remark}

\begin{remark}\label{2Duality}
It is clear that the adjoint $T^*\colon X^*\to X^*$
of a Ritt operator $T\in B(X)$ (of type $\alpha$) is a 
Ritt operator (of type $\alpha$) as well. In this case, $\varphi(T)^*=
\varphi(T^*)$ for any $\varphi\in H^{\infty}_0(B_\gamma)$ with $\gamma>\alpha$. Hence
$T^*$ has a bounded $H^{\infty}(B_\gamma)$ functional calculus if and only if
$T$ has one.
\end{remark}

\medskip
\section{Square functions}
On general Banach spaces, square functions of the form (\ref{1SFH}) or (\ref{1SFLp})
need to  be replaced by suitable Rademacher averages.  This short section is devoted to 
precise definitions of these abstract square functions, 
as well as to relevant properties of Rademacher norms
on certain Banach spaces.

We let $(\varepsilon_k)_{k\geq 1}$ 
be a sequence of independent Rademacher variables on some probability space
$(\M, d\Pdb)$. Given any Banach space $X$, we let ${\rm Rad}(X)$ denote the
closed subspace of the Bochner space $L^{2}(\M ;X)$ spanned 
by the set $\{\varepsilon_k\otimes x\, :\, k\geq 1, x\in X\}$. Thus for any
finite family $(x_k)_{k\geq 1}$ of elements of $X$,
\begin{equation}\label{2Rad}
\Bignorm{\sum_k \varepsilon_k\otimes x_k}_{{\rm Rad}(X)}\,
=\,\biggl(\int_{\footnotesize{\M}}\Bignorm{\sum_k 
\varepsilon_k(u) x_k}_{X}^{2}\, d\Pdb(u)\,\biggr)^{\frac{1}{2}}\,.
\end{equation}
Moreover elements of ${\rm Rad}(X)$ are sums of convergent series of the form 
$\sum_{k=1}^{\infty}\varepsilon_k\otimes x_k\,$.

For any bounded operator $T\colon X\to X$,
for any integer $m\geq 1$ and for any $x\in X$, we set
$$
\norm{x}_{T,m}\,=\,\biggnorm{\sum_{k=1}^{\infty} k^{m-\frac{1}{2}}\,
\varepsilon_k\otimes T^{k-1}(I-T)^m (x)}_{{\rm Rad}(X)}.
$$
More precisely for any $x\in X$ and any integer $k\geq 1$, set 
$x_k= k^{m-\frac{1}{2}}T^{k-1}(I-T)^m (x)$. Then 
$\norm{x}_{T,m}$ is equal to the ${\rm Rad}(X)$-norm of $\sum_{k=1}^{\infty} 
\varepsilon_k\otimes x_k$ if this series converges in $L^2(\M;X)$, and 
$\norm{x}_{T,m}=\infty$ otherwise.

If $X=L^p(\Omega)$ for some $1\leq p<\infty$, then we have an equivalence
\begin{equation}\label{2K}
\Bignorm{\sum_k \varepsilon_k\otimes x_k}_{{\rm Rad}(L^p(\Omega))}\,
\approx\,\Bignorm{\Bigl(\sum_k\vert x_k\vert^2\Bigr)^{\frac{1}{2}}}_{L^p(\Omega)}
\end{equation}
for finite families $(x_k)_k$ of $X$ (see e.g. \cite[Thm. 1.d.6]{LT2}). Hence for any 
$T\colon L^p(\Omega)\to L^p(\Omega)$ and any $m\geq 1$, we have
\begin{equation}\label{2SFLp}
\norm{x}_{T,m}\,\approx\,\Bignorm{\Bigl(
\sum_{k=1}^{\infty} k^{2m-1}\bigl\vert T^{k-1}(I-T)^m(x)\bigr\vert^2\Bigr)^{\frac{1}{2}}}_{L^p(\Omega)},
\qquad x\in L^p(\Omega).
\end{equation}
In particular, the square function $\norm{\cdotp}_T$ defined by (\ref{1SFLp}) is equivalent
to $\norm{\cdotp}_{T,1}$.

Likewise, the Rademacher average (\ref{2Rad}) of a finite sequence $(x_k)_k$ on Hilbert space 
$H$ is equal to $\bigl(\sum_k\norm{x_k}_H^2\bigr)^{\frac{1}{2}}$, hence for 
any $T\in B(H)$, we have
$$
\norm{x}_{T,m}\,=\,\Bigl(\sum_{k=1}^{\infty}k^{2m-1}\bignorm{T^{k-1}(I-T)^m(x)}^2\Bigr)^{\frac{1}{2}},
\qquad x\in H.
$$

Square functions appearing in (\ref{2SFLp}) are analogs of well-known 
square functions associated to sectorial operators
on $L^p$-spaces. Namely let $A$ be a sectorial operator 
of type $<\frac{\pi}{2}$ on $L^p(\Omega)$. Then 
$-A$ generates a bounded analytic semigroup $(e^{-tA})_{t\geq 0}$ 
on $L^p(\Omega)$
and for any integer $m\geq 1$, one may consider
$$
\norm{x}_{A,m}\,=\,\biggnorm{\biggl(\int_{0}^{\infty} t^{2m-1} \bigl\vert A^m e^{-tA}
(x)\bigr\vert^2
\,dt\,\biggr)^{\frac{1}{2}}}_{L^p(\Omega)},\qquad x\in L^p(\Omega).
$$
For any $t>0$, $\frac{\partial^m}{\partial t^m}\bigl(e^{-tA}\bigr)=(-1)^mA^m
e^{-tA}$. Hence if we regard $\bigl(T^{k-1}(I-T)^m\bigr)_{k\geq 1}$ as the 
$m$-th discrete derivative of the sequence $(T^{k-1})_{k\geq 1}$, then
$\norm{\cdotp}_{T,m}$ is the discrete analog of the continuous square 
function $\norm{\cdotp}_{A,m}$. Thus 
Theorem \ref{1MainLp} is a discrete analog of the 
main result of \cite{CDMY} showing the equivalence 
between the boundedness of $H^\infty$-functional calculus and 
square function estimates for sectorial operators. 

Similar comments apply to the Hilbert space case. 

In the sequel, the square functions $\norm{\cdotp}_{T,m}$ will 
be used for Ritt operators (although their definitions make sense
for any operator).

\bigskip
Let $X$ be a Banach space. The space ${\rm Rad}({\rm Rad}(X))$ is the closure 
of finite sums
$$
\sum_{i,j}\varepsilon_i\otimes\varepsilon_j\otimes x_{ij}
$$
in $L^2(\M\times\M;X)$, where $x_{ij}\in X$ for any $i,j\geq 1$. We say that 
$X$ has property $(\alpha)$ if the above decomposition is unconditional, that is,
there exists a constant $C>0$ such that for any finite family $(x_{ij})_{i,j\geq 1}$
of $X$ and any family $(t_{ij})_{i,j\geq 1}$ of complex numbers,
$$
\Bignorm{\sum_{i,j} \varepsilon_i\otimes\varepsilon_j\otimes t_{ij}\,x_{ij}}_{{\rm Rad}({\rm Rad}(X))}
\,\leq\, C\,\sup_{i,j}\vert t_{ij}\vert\,
\Bignorm{\sum_{i,j} \varepsilon_i\otimes\varepsilon_j\otimes x_{ij}}_{{\rm Rad}({\rm Rad}(X))}.
$$
Classical $L^p$-spaces (for $p<\infty$) have property $(\alpha)$, indeed we have an
equivalence 
\begin{equation}\label{2KK}
\Bignorm{\sum_{i,j} \varepsilon_i\otimes\varepsilon_j\otimes x_{ij}}_{{\rm Rad}({\rm Rad}(L^p(\Omega)))}\,
\approx\,\Bignorm{\Bigl(\sum_{i,j}\vert x_{ij}\vert^2\Bigr)^{\frac{1}{2}}}_{L^p(\Omega)}
\end{equation}
for finite families $(x_{ij})_{i,j}$ of $L^p(\Omega)$, which extends (\ref{2K}).
This actually holds true as well for any Banach lattice with a finite
cotype in place of $L^p(\Omega)$.

On the contrary, infinite dimensional noncommutative $L^p$-spaces (for $p\not=2$) do not have
property $(\alpha)$. This goes back to \cite{P0}, where property $(\alpha)$ was introduced.

We shall now supply more precise information, namely the so-called noncommutative Khintchine
inequalilites in one or two variables. In the one-variable case, these inequalities, stated
as (\ref{2KI1}) and (\ref{2KI2}) below are due to Lust-Piquard for $1<p<\infty$ \cite{LP}
and Lust-Piquard and Pisier for $p=1$ \cite{LPP}. The two-variable inequalities
(\ref{2KI3}) and (\ref{2KI4}) are taken from \cite[pp. 111-112]{P1}.

In the sequel we let $M$
be a semifinite von Neumann algebra equipped with a normal semifinite faithful trace and
for any $1\leq p\leq \infty$, we let $L^p(M)$ denote the associated noncommutative $L^p$-space. We
refer the reader  to \cite{PX} for background and general information on these spaces. 
Any element of $L^p(M)$ is a
(possibly unbounded) operator and for any such $x$, the modulus of $x$ used in the next formulas will be
$$
\vert x\vert =(x^*x)^{\frac{1}{2}}.
$$
The following equivalences, valid for finite families of $L^p(M)$,
are the noncommutative counterpart of (\ref{2K}). If 
$2\leq p<\infty$, then
\begin{equation}\label{2KI1}
\Bignorm{\sum_k \varepsilon_k\otimes x_k}_{{\rm Rad}(L^p(M))}\approx\max\Bigl\{\Bignorm{\Bigl(\sum_k\vert x_k\vert^2\Bigr)^{\frac{1}{2}}}_{L^p(M)},\,
\Bignorm{\Bigl(\sum_k\vert x_k^*\vert^2\Bigr)^{\frac{1}{2}}}_{L^p(M)}\Bigr\}.
\end{equation}
If $1\leq p\leq 2$, then
\begin{equation}\label{2KI2}
\Bignorm{\sum_k \varepsilon_k\otimes x_k}_{{\rm Rad}(L^p(M))}\approx\inf\Bigl\{\Bignorm{\Bigl(\sum_k\vert u_k\vert^2\Bigr)^{\frac{1}{2}}}_{L^p(M)} + 
\Bignorm{\Bigl(\sum_k\vert v_k^*\vert^2\Bigr)^{\frac{1}{2}}}_{L^p(M)}\Bigr\},
\end{equation}
where the infimum runs over all possible decompositions
$x_k=u_k+v_k$ in $L^p(M)$.

Let $n\geq 1$ be an integer. The space $L^p(M_n(M))$ associated to 
the von Neumann algebra $M_n(M)$ can be canonically identified with 
the vector space of all $n\times n$
matrices with entries in $L^p(M)$.
The following equivalences 
are the noncommutative counterpart of (\ref{2KK}). 
If $2\leq p<\infty$, then
\begin{align}\label{2KI3}
\Bignorm{\sum_{i,j=1}^n \varepsilon_i\otimes\varepsilon_j\otimes x_{ij}}_{{\rm Rad}({\rm Rad}(L^p(M)))} 
\approx\max\Bigl\{&\Bignorm{\Bigl(\sum_{i,j=1}^n \vert x_{ij}\vert^2\Bigr)^{\frac{1}{2}}}_{L^p(M)},\,
\Bignorm{\Bigl(\sum_{i,j=1}^n\vert x_{ij}^*\vert^2\Bigr)^{\frac{1}{2}}}_{L^p(M)},\\
&\bignorm{[x_{ij}]}_{L^p(M_n(M))},\, \bignorm{[x_{ji}]}_{L^p(M_n(M))}
\Bigr\}.\notag
\end{align}
If $1\leq p\leq 2$, then
\begin{align}\label{2KI4}
\Bignorm{\sum_{i,j=1}^n \varepsilon_i\otimes\varepsilon_j\otimes x_{ij}}_{{\rm Rad}({\rm Rad}(L^p(M)))} 
\approx\inf\Bigl\{&\Bignorm{\Bigl(\sum_{i,j=1}^n \vert u_{ij}\vert^2\Bigr)^{\frac{1}{2}}}_{L^p(M)}\, +\,
\Bignorm{\Bigl(\sum_{i,j=1}^n\vert v_{ij}^*\vert^2\Bigr)^{\frac{1}{2}}}_{L^p(M)}\,\\
&+\,\bignorm{[w_{ij}]}_{L^p(M_n(M))}\,+\, \bignorm{[z_{ji}]}_{L^p(M_n(M))}
\Bigr\},\notag
\end{align}
where the infimum runs over all possible decompositions
$x_{ij}=u_{ij}+v_{ij}+ w_{ij}+z_{ij}$ in $L^p(M)$.

\medskip
\section{A transfer principle from sectorial operators to Ritt operators}
Let $T\colon X\to X$ be a Ritt operator on an arbitrary Banach space.
We noticed in Section 2 that 
$$
A=I-T
$$
is a sectorial operator of type $<\frac{\pi}{2}$.
The following transfer result will be extremely important for applications.
Indeed it allows to apply known results from the theory of $H^\infty$-calculus 
for sectorial operators to our context. This 
principle will be illustrated in Section 8.
The proof is a variant of the one of \cite[Thm. 8.3]{Ha0}, adapted to our
situation (see also 
\cite[Prop. 3.2]{LMX1}).

\begin{proposition}\label{4TA}
The following are equivalent.
\begin{enumerate}
\item [(i)] $T$ admits a bounded $H^{\infty}(B_\gamma)$ functional
calculus for some $\gamma\in \bigl(0,\frac{\pi}{2}\bigr)$.
\item [(ii)] $A$ admits a bounded $H^{\infty}(\Sigma_\theta)$ functional
calculus for some $\theta\in \bigl(0,\frac{\pi}{2}\bigr)$.
\end{enumerate}
\end{proposition}

\begin{proof}
It will be convenient to set
$$
\Delta_\gamma = 1-B_\gamma.
$$
for any $\gamma\in\bigl(0,\frac{\pi}{2}\bigr)$. This is a subset of the cone $\Sigma_\gamma$.

Assume (i). To any $f\in H^{\infty}_0(\Sigma_\gamma)$, associate 
$\varphi$ given by $\varphi(\lambda)=f(1-\lambda)$.
Then $\varphi$ is defined on $B_\gamma$, its restriction to that set belongs to 
$H_{0}^{\infty}(B_\gamma)$, and $\norm{\varphi}_{\infty,B_\gamma}=\norm{f}_{\infty,\Delta_\gamma}
\leq\norm{f}_{\infty,\Sigma_\gamma}$. 
Comparing (\ref{2CauchySec}) and (\ref{2Cauchy}) and applying
Cauchy's Theorem, we see that 
$$
f(A)=\varphi(T).
$$
These observations imply that $A$  has a bounded $H^{\infty}(\Sigma_\gamma)$ functional
calculus.

Assume conversely that $A$ admits a bounded $H^{\infty}(\Sigma_\theta)$ functional
calculus for some $\theta$ in $\bigl(0,\frac{\pi}{2}\bigr)$. It follows from
Lemma \ref{2Blunck} that
$$
\sigma(A)\subset\overline{\Delta_\alpha}
$$
for some $\alpha\in \bigl(0,\frac{\pi}{2}\bigr)$. 
Taking $\theta$ close enough to 
$\frac{\pi}{2}$, we may assume that $\alpha<\theta$.

We fix $\gamma\in\bigl(\theta,\frac{\pi}{2}\bigr)$ and choose an arbitrary 
$\beta\in (\theta,\gamma)$. Let $\Gamma_1$ be the juxtaposition of the
segments $[\cos(\beta)e^{i\beta}, 0]$ and $[0, \cos(\beta)e^{-i\beta}]$.
Then let $\Gamma_2$ be the curve going from $\cos(\beta)e^{-i\beta}$
to $\cos(\beta)e^{i\beta}$ counterclockwise along the circle of center 1 and
radius $\sin(\beta)$. Thus 
\begin{equation}\label{4Juxt}
\partial \Delta_\beta\,=\,\bigl\{\Gamma_1,\Gamma_2\bigr\},
\end{equation}
the juxtaposition of $\Gamma_1$ and $\Gamma_2$ (see Figure 2 below).

\begin{figure}[ht]
\begin{center}
\includegraphics[scale=0.4]{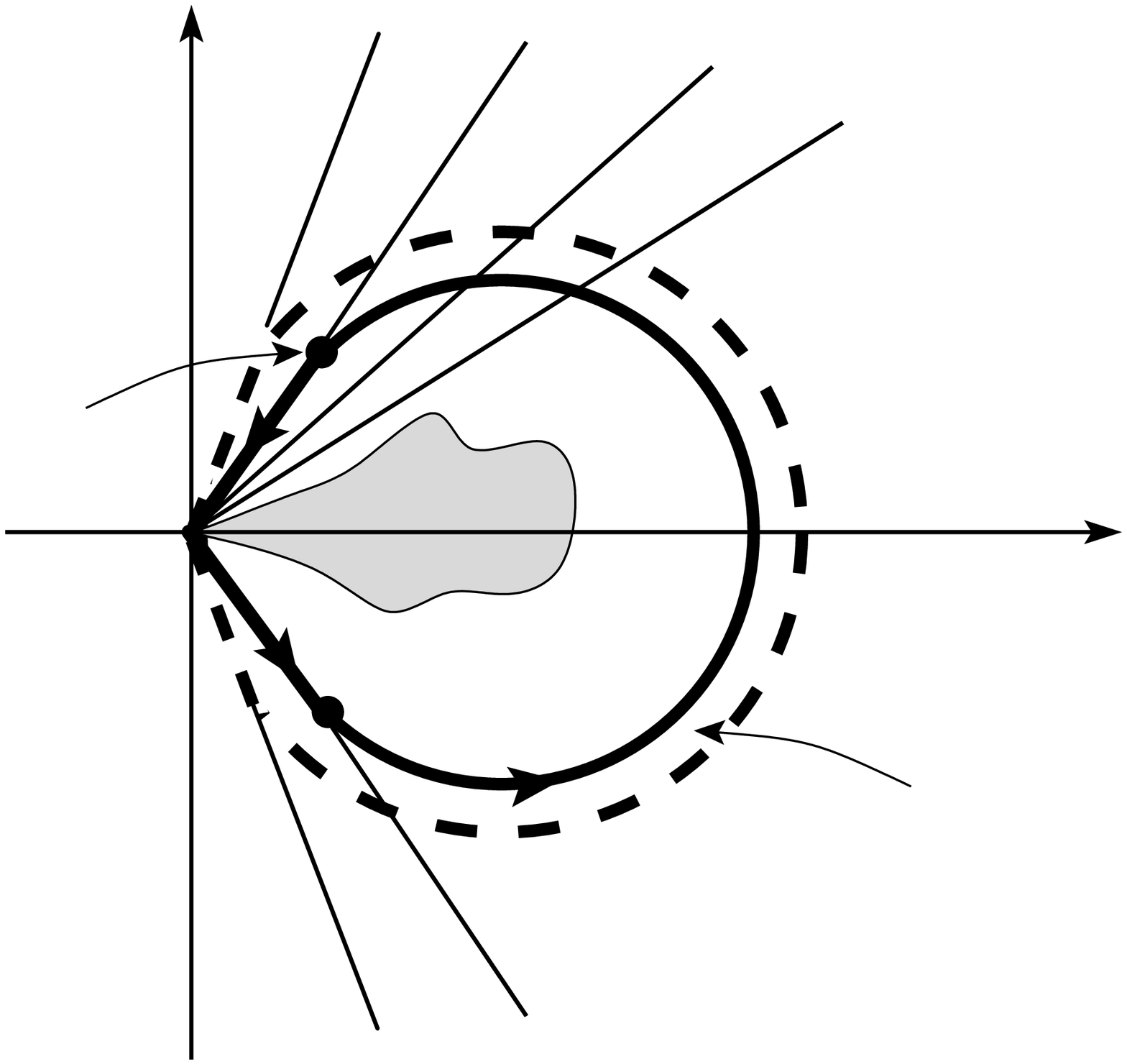}
\begin{small}
\begin{picture}(0,0) 
\put(-143,190){$\gamma$}
\put(-113,186){$\beta$}
\put(-80,181){$\theta$}
\put(-56,171){$\alpha$} 
\put(-47,43){$\Delta_\gamma$} 
\put(-158,76){$\Gamma_1$} 
\put(-112,61){$\Gamma_2$} 
\put(-148,102){$\sigma(A)$}
\put(-245,116){$\cos(\beta){\rm e}^{\rm i\beta}$}
\put(-187,88){$0$}  
\end{picture}
\end{small}
\end{center}
\caption{\label{f2}}
\end{figure}

Let $\varphi\in H_{0}^{\infty}(B_\gamma)$ and let $f\colon \Delta_\gamma\to\Cdb$ be 
the holomorphic function defined by
\begin{equation}\label{4Varphi}
f(z)=\varphi(1-z),\qquad z\in \Delta_\gamma.
\end{equation}
Then again we have $\norm{f}_{\infty,\Delta_\gamma} = \norm{\varphi}_{\infty, B_\gamma}$,
moreover there exist two positive constants $c,s>0$ such that
\begin{equation}\label{4H0}
\vert f(z)\vert\leq c\vert z\vert^s,\qquad z\in \Delta_\gamma.
\end{equation}

We may define 
$f_1\colon\Cdb\setminus\Gamma_1\to \Cdb\,$ and 
$f_2\colon\Cdb\setminus\Gamma_2\to \Cdb\,$ by letting
\begin{equation}\label{4f1f2}
f_1(z)\,=\,\frac{1}{2\pi i}\,\int_{\Gamma_1}\frac{f(\lambda)}{\lambda-z}\, d\lambda
\qquad\hbox{and}\qquad
f_2(z)\,=\,\frac{1}{2\pi i}\,\int_{\Gamma_2}\frac{f(\lambda)}{\lambda-z}\, d\lambda\,.
\end{equation}
Clearly these functions are holomorphic on their domains.
According to $(\ref{4Juxt})$ and Cauchy's Theorem, we have
\begin{equation}\label{4Cauchy}
\forall\, z\in \Delta_\beta,\qquad f(z)=f_1(z)+f_2(z).
\end{equation}
Since the distance between $\Gamma_1$ and $\Sigma_\theta\setminus\Delta_\theta$ 
is strictly positive and $\Gamma_1\subset \Delta_\gamma$, 
there is a constant $C_1\geq 0$ (not depending on $f$) such that
\begin{equation}\label{4C1}
\forall\, z\in \Sigma_\theta\setminus\Delta_\theta,\qquad \vert f_1(z)\vert\leq C_1
\norm{f}_{\infty,\Delta_\gamma}.
\end{equation}
Likewise there is a constant $C_2\geq 0$ (not depending on $f$) such that
$$
\forall\, z\in \Delta_\theta,\qquad \vert f_2(z)\vert\leq C_2
\norm{f}_{\infty,\Delta_\gamma}.
$$
Combining with (\ref{4Cauchy}), this yields
$$
\forall\, z\in \Delta_\theta,\qquad \vert f_1(z)\vert\leq (1+C_2)
\norm{f}_{\infty,\Delta_\gamma}.
$$
Together with (\ref{4C1}) this shows that $f_1\in H^{\infty}(\Sigma_\theta)$
and that with $C_3=\max\{C_1, 1+C_2\}$, we have 
\begin{equation}\label{4Key}
\norm{f_1}_{\infty,\Sigma_\theta}\leq C_3\norm{f}_{\infty,\Delta_\gamma}.
\end{equation}
Now let $g\colon \Sigma_\theta\to\Cdb$ be defined by
$$
g(z) = f_1(z) +\,\frac{f_2(0)}{1+z}. 
$$
According to the definition of $f_1$ given by (\ref{4f1f2}), $zf_1(z)$ is bounded when $\vert z\vert\to\infty$.
Hence $zg(z)$ is bounded on $\Sigma_\theta$. Further, $f_2$ is defined 
about $0$, hence $\vert f_2(z) -f_2(0)\vert\lesssim\vert z\vert$ on $\Delta_\theta$.
By (\ref{4Cauchy}), we have 
$$
g(z) = f(z) + \Bigl(\frac{f_2(0)}{1+z} -f_2(z)\Bigr) = f(z) + \bigl(f_2(0) - f_2(z)\bigr) \,
-\,f_2(0)\,\frac{z}{1+z}
$$
on $\Delta_\theta$.
Applying the above estimate and (\ref{4H0}), we deduce that
$\vert g(z)\vert\lesssim \max\{\vert z\vert^s,\vert z\vert\}$ on $\Delta_\theta$.
These estimates show that $g$ belongs to 
$H^{\infty}_{0}(\Sigma_\theta)$. 
We may therefore compute $g(A)$ by means of (\ref{2CauchySec}), and hence $f_1(A)$ by
$$
f_1(A) = g(A)\,-\,f_2(0)(I+A)^{-1}.
$$
From the assumption (ii), we get a constant $C_4\geq 0$ (not depending on $f$) such that
$$
\norm{f_1(A)}\leq C_4\norm{f_1}_{\infty,\Sigma_\theta}.
$$
Combining with (\ref{4Key}), we deduce
$$
\norm{f_1(A)}\leq C_3 C_4\norm{f}_{\infty,\Delta_\gamma}.
$$
The holomorphic function $f_2$ is defined on an open neighborhood of
the spectrum $\sigma(A)$. Hence $f_2(A)$ may be defined 
by the classical Riesz-Dunford functional calculus. Then by Fubini's Theorem
and (\ref{4f1f2}), we have
$$
f_2(A)\,=\,\frac{1}{2\pi i}\,\int_{\Gamma_2} f(\lambda)R(\lambda,A)\, d\lambda\,.
$$
Consequently,
$$
\norm{f_2(A)}\,\leq\,\frac{1}{2\pi}\,\int_{\Gamma_2}\vert f(\lambda)\vert\norm{R(\lambda,A)}\, \vert d\lambda\vert
\,.
$$
We deduce that there is a constant $C_5\geq 0$ (not depending on $f$) such that
$$
\norm{f_2(A)}\,\leq C_5\norm{f}_{\infty,\Delta_\gamma}.
$$
Using (\ref{4Varphi}) and (\ref{4Cauchy}), it is easy to check that
$$
\varphi(T)\,=\,f_1(A)+f_2(A).
$$
We deduce (with $C=C_3C_4 +C_5$) an estimate
$$
\norm{\varphi(T)}\leq C\norm{\varphi}_{\infty, B_\gamma},
$$
which shows the boundedness of the $H^\infty(B_\gamma)$ functional calculus.
\end{proof}

\begin{remark}
We mention another (easier) transfer principle. Let $(T_t)_{t\geq 0}$
be a bounded analytic semigroup on $X$, and let $-A$ denote its infinitesimal 
generator. For any fixed $t\geq 0$,
$T_t$ is a Ritt operator; this is easy to check, see \cite[Section 3]{V2} for 
more on this. Writing $\varphi(T_t)=f(A)$ with $f(z)=\varphi(e^{-tz})$, one shows
that if $A$ admits a bounded $H^{\infty}(\Sigma_\theta)$
functional calculus for some $\theta\in\bigl(0,\frac{\pi}{2}\bigr)$, then
$T_t$ admits a bounded $H^\infty(B_\gamma)$ functional calculus for some
$\gamma\in\bigl(0,\frac{\pi}{2}\bigr)$.
\end{remark}

\medskip
\section{$R$-boundedness and $R$-Ritt operators}
This section starts with some background on $R$-boundedness, a notion
which -by now- plays a prominent role in many questions concerning functional calculi,
see in particular \cite{KaW1, KaW2, W1}. $R$-boundedness was introduced
in  \cite{BG} and significantly developed in \cite{CPSW}. The resulting 
notion of $R$-Ritt operator (see below) was first studied by
Blunck \cite{Bl1,Bl2}. 

Let $X$ be a Banach space and let $E\subset B(X)$ be a set of 
bounded operators on $X$. We say that $E$ is $R$-bounded provided that
there exists a constant $C\geq 0$ such that for any finite family
$(T_k)_k$ of $E$ and any finite family $(x_k)_k$ of $X$,
$$
\Bignorm{\sum_k\varepsilon_k\otimes T_k(x_k)}_{{\rm Rad}(X)}\,
\leq\,C
\Bignorm{\sum_k\varepsilon_k\otimes x_k}_{{\rm Rad}(X)}.
$$
In this case, we let $\R(E)$ denote the smallest possible $C$.
Any $R$-bounded set $E$ is bounded, with $\norm{T}\leq\R(E)$ for any
$T\in E$. If $X=H$ is a Hilbert space, the 
converse holds true, because of the isometric isomorphism
${\rm Rad}(H)=\ell^2(H)$. But if $X$ is not isomorphic
to a Hilbert space, then the unit ball of $B(X)$ is not $R$-bounded
(see \cite{AB}).

We will use the following convexity result.
This is a well-known consequence of \cite[Lem. 3.2]{CPSW}, see also \cite[Lem. 4.2]{JLX}.

\begin{lemma}\label{5Convexity} Let $J\subset \Rdb$ be an interval, let
$E\subset B(X)$ be an $R$-bounded set and let $K>0$ be a constant. Then the set
$$
E_K\,=\,\biggl\{\int_J h(t)F(t)\, dt\ \Bigl\vert\, F\colon J\to E\ \hbox{is continuous},\
h\in L^1(J;dt)\ \hbox{and}\ \int_J\vert h(t)\vert\, dt\,\leq K\biggr\}
$$
is $R$-bounded, with $\R(E_K)\leq 2K\R(E)$.
\end{lemma}

A sectorial operator $A$ on $X$ is called $R$-sectorial of 
$R$-type $\omega$ provided that $\sigma(A)\subset\overline{\Sigma_\omega}$ 
and for any $\nu\in(\omega,\pi)$, the set 
(\ref{2Sectorial}) is $R$-bounded. 

Likewise, a Ritt operator $T$ on $X$ is called $R$-Ritt 
provided that the two sets in (\ref{2Sets}) are $R$-bounded.
The following is an `$R$-bounded' version of (\ref{2Ritt}) and 
Lemma \ref{2Blunck}.
We refer to \cite{Bl1} for closely related results.

\begin{lemma}\label{5Blunck} Let $T\colon X\to X$ be a Ritt operator and
let $A=I-T$. The following are equivalent.
\begin{itemize}
\item [(i)] $T$ is $R$-Ritt.
\item [(ii)] $A$ is $R$-sectorial of $R$-type $<\frac{\pi}{2}$.
\item [(iii)] There exists an angle $\alpha\in \bigl(0,\frac{\pi}{2}\bigr)$ such that
$\sigma(T)\subset \overline{B_\alpha}$
and for any $\beta\in\bigl(\alpha,\frac{\pi}{2}\bigr)$, the set
\begin{equation}\label{2Blunck2}
\bigl\{(\lambda-1)R(\lambda,T)\, :\, \lambda\in\Cdb\setminus 
\overline{B_\beta} \bigr\}
\end{equation}
is $R$-bounded. 
\end{itemize}
\end{lemma}

\begin{proof}
The implications `(i)$\Rightarrow$(ii)' and  `(iii)$\Rightarrow$(i)'
follow from \cite{Bl1}. The proof of `(ii)$\Rightarrow$(iii)'
is parallel to the one of Lemma \ref{2Blunck}, using
two elementary but important results on $R$-boundedness due to L. Weis. The
first one says that for any open set $\O\subset\Cdb$ and for any compact set 
$F\subset \O$, any analytic function $\O\to B(X)$ 
maps $F$ into an $R$-bounded subset of $B(X)$ \cite[Prop. 2.6]{W1}.
With the notation of the proof of Lemma \ref{2Blunck}, this implies that
the two sets
$$
E_1=h(\Lambda_\beta)
\qquad\hbox{and}\qquad 
E_2=\{h(\lambda)\, :\,\vert\lambda\vert=2\}
$$
are $R$-bounded. The second one is the `maximum principle' for 
$R$-boundedness \cite[Prop. 2.8]{W1}. Together with the $R$-boundedness
of $E_2$, it implies that 
$\{h(\lambda)\, :\,\vert\lambda\vert\geq 2\}$ is $R$-bounded. 
With these elements in hand, the adaptation of the proof 
of Lemma \ref{2Blunck} is straightforward.
\end{proof}

We will say that $T$ is an $R$-Ritt operator of $R$-type $\alpha$ if it satisfies condition
(iii) of Lemma \ref{5Blunck}. It is clear that in this case, $A=I-T$ is
$R$-sectorial of $R$-type $\alpha$.

In the rest of this section, 
we are going to focus on {\it commutative} $L^p$-spaces,
see however Remark \ref{5Alpha}.
Our objective is the following theorem, which is a key step in
our proof of Theorem \ref{1MainLp}.

\begin{theorem}\label{5Automatic}
Let $(\Omega,\mu)$ be a measure space, let $1<p<\infty$ and let 
$T\colon L^p(\Omega)\to L^p(\Omega)$ be a power bounded operator. Assume that
it satisfies uniform estimates
\begin{equation}\label{5DbSFE}
\norm{x}_{T,1}\lesssim\norm{x}_{L^p}\qquad \hbox{and} \qquad
\norm{y}_{T^*,1}\lesssim\norm{y}_{L^{p'}}
\end{equation}
for $x\in L^p(\Omega)$ and $y\in L^{p'}(\Omega)$.
Then $T$ is $R$-Ritt.
\end{theorem}

Until the end of the proof of this theorem, we fix a bounded operator
$T\colon L^p(\Omega)\to L^p(\Omega)$, with $1<p<\infty$.
The following lemma is inspired by the proof of \cite[Thm. 4.7]{KP}.

\begin{lemma}\label{512}
If $T$ satisfies a uniform estimate
\begin{equation}\label{51}
\norm{x}_{T,1}\lesssim\norm{x},\qquad x\in L^p(\Omega),
\end{equation}
then it automatically satisfies a uniform estimate
\begin{equation}\label{52}
\norm{x}_{T,2}\lesssim\norm{x},\qquad x\in L^p(\Omega).
\end{equation}
\end{lemma}

\begin{proof}
We will use the following elementary identity that the reader can easily check. 
For any integer $k\geq 1$,
\begin{equation}\label{KP1}
\sum_{j=1}^{k} j(k+1-j)\,=\,\frac{1}{6}\,k(k+1)(k+2).
\end{equation}
Let $x\in L^p(\Omega)$ and let $N\geq 1$ be an integer.
According to the above identity we have a function inequality
$$
\sum_{k=1}^{N} k^3\,\bigl\vert T^{k-1}(I-T)^2x\bigr\vert^2
\,\leq\,
6\sum_{k=1}^{N}\sum_{j=1}^{k} j(k+1-j)\bigl\vert T^{k-1}(I-T)^2x\bigr\vert^2.
$$
By a change of indices (letting $r=k+1-j$ for any fixed $j$), we have
\begin{align*}
\sum_{k=1}^{N}\sum_{j=1}^{k} j(k+1-j)\bigl\vert T^{k-1}(I-T)^2x\bigr\vert^2\,
&=\,\sum_{j=1}^{N} j\sum_{k=j}^{N}(k+1-j)\bigl\vert T^{k-1}(I-T)^2x\bigr\vert^2\\
&=\,\sum_{j=1}^{N} j\sum_{r=1}^{N+1-j} r\bigl\vert T^{r+j-2}(I-T)^2x\bigr\vert^2\\
&\leq\,\sum_{j=1}^{N} j\sum_{r=1}^{N} r\bigl\vert T^{r+j-2}(I-T)^2x\bigr\vert^2.
\end{align*}
According to (\ref{2KK}), we have an estimate
$$
\Bignorm{\Bigl(\sum_{j,r=1}^{N}  jr \bigl\vert T^{r+j-2}(I-T)^2x\bigr\vert^2\Bigr)^{\frac{1}{2}}}_{L^p(\Omega)}\,
\lesssim\,
\Bignorm{\Bigl(\sum_{j,r=1}^{N}j^{\frac{1}{2}}r^{\frac{1}{2}}\, \varepsilon_j\otimes \varepsilon_r\otimes 
 T^{r+j-2}(I-T)^2x}_{{\rm Rad}({\rm Rad}(L^p(\Omega)))}.
$$
Furthermore, writing 
$$
T^{r+j-2}(I-T)^2x\, =\,T^{j-1}(I-T)\bigl[T^{r-1}(I-T) x\bigr],
$$
and applying the assumption (\ref{51}) twice, we see that
\begin{align*}
\Bignorm{\sum_{j,r=1}^{N} j^{\frac{1}{2}}r^{\frac{1}{2}}\varepsilon_j\otimes \varepsilon_r\otimes 
T^{r+j-2}(I-T)^2x}_{{\rm Rad}({\rm Rad}(L^p(\Omega)))}\,
&
\lesssim
\,\Bignorm{\sum_{r=1}^{N} r^{\frac{1}{2}}\varepsilon_r \otimes 
T^{r-1}(I-T)x}_{{\rm Rad}(L^p(\Omega))}\\
&\lesssim
\,\norm{x}.
\end{align*}
Altogether, we obtain the estimate
$$
\Bignorm{\Bigl(\sum_{k=1}^{N} k^3\,\bigl\vert T^{k-1}(I-T)^2x\bigr\vert^2
\Bigr)^{\frac{1}{2}}}_{L^p(\Omega)}\,\lesssim\,\norm{x},
$$
which proves (\ref{52}).
\end{proof}

\begin{proof}[Proof of Theorem \ref{5Automatic}]
Since $T$ is power bounded and $X=L^p(\Omega)$ is reflexive, 
the Mean Ergodic Theorem ensures that 
\begin{equation}\label{5MET}
X=\, {\rm Ker}(I-T)\oplus \overline{{\rm Ran}(I-T)}.
\end{equation}
Furthermore the two square function estimates (\ref{5DbSFE}) imply that
$$
\norm{x}\approx\norm{x}_{T,1},\qquad x\in \overline{{\rm Ran}(I-T)}.
$$
Indeed this is implicit in \cite[Cor. 3.4]{LMX1}, to which we refer for details.
Let $(x_n)_{n\geq 1}$ be a finite family of $\overline{{\rm Ran}(I-T)}$, and let
$(\eta_n)_{n\geq 1}$ be a sequence of $\pm 1$. The above equivalence yields
$$
\Bignorm{\sum_{n\geq 1} \eta_n \, x_n}_{L^p(\Omega)}\,\approx\,\Bignorm{\sum_{k\geq 1}\sum_{n\geq 1} k^{\frac{1}{2}}\eta_n\,
\varepsilon_k\otimes T^{k-1}(I-T)x_n}_{{\rm Rad}(L^p(\Omega))}.
$$
Averaging over the $\eta_n=\pm 1$ and applying (\ref{2KK}), we obtain that
\begin{equation}\label{5Equiv}
\Bignorm{\sum_{n\geq 1} \varepsilon_n \otimes x_n}_{{\rm Rad}(L^p(\Omega))}\,\approx\,
\Bignorm{\Bigl(\sum_{k,n\geq 1} k\bigl\vert 
T^{k-1}(I-T)x_n\bigr\vert^2\Bigr)^{\frac{1}{2}}}_{L^p(\Omega)}
\end{equation}
for $x_n$ in $\overline{{\rm Ran}(I-T)}$. 

Applying Lemma \ref{512} and similarly 
averaging the resulting estimates
$$
\Bignorm{\sum_n\eta_n\, x_n}_{T,2}\,\lesssim\,\Bignorm{\sum_n\eta_n\, x_n}
$$
over all $\eta_n=\pm 1$, we obtain that
\begin{equation}\label{52Ave}
\Bignorm{\Bigl(\sum_{k,n\geq 1} k^3\bigl\vert T^{k-1}(I-T)^2
x_n\bigr\vert^2 \Bigr)^{\frac{1}{2}}}_{L^p(\Omega)}\,\lesssim\,
\Bignorm{\sum_{n\geq 1} \varepsilon_n \otimes x_n}_{{\rm Rad}(L^p(\Omega))}
\end{equation}
for $x_n$ in $L^p(\Omega)$.

Our aim is to show that the two sets in (\ref{2Sets}) are $R$-bounded. Their restrictions to the kernel
${\rm Ker}(I-T)$ clearly have this property. By (\ref{5MET}) it therefore suffices to
consider their restrictions to  $\overline{{\rm Ran}(I-T)}$.

Let $(x_n)_{n\geq 1}$ be a finite family of $\overline{{\rm Ran}(I-T)}$. 
Each $T^nx_n$ belongs to that
space hence by (\ref{5Equiv}), we have
$$
\Bignorm{\sum_{n\geq 1} \varepsilon_n \otimes T^nx_n}_{{\rm Rad}(L^p(\Omega))}\,
\lesssim\,
\Bignorm{\Bigl(\sum_{k,n\geq 1} k\bigl\vert T^{k+n-1}(I-T)x_n\bigr\vert^2\Bigr)^{\frac{1}{2}}}_{L^p(\Omega)}.
$$
Moreover
\begin{align*}
\Bignorm{\Bigl(\sum_{k,n\geq 1} k\bigl\vert T^{k+n-1}(I-T)x_n\bigr\vert^2\Bigr)^{\frac{1}{2}}}_{L^p(\Omega)}\,
&\leq\, 
\Bignorm{\Bigl(\sum_{k,n\geq 1} (k+n)\bigl\vert T^{k+n-1}(I-T)x_n\bigr\vert^2\Bigr)^{\frac{1}{2}}}_{L^p(\Omega)}\\
&\leq\, 
\Bignorm{\Bigl(\sum_{n\geq 1}\sum_{k\geq n+1} k\bigl\vert T^{k-1}(I-T)x_n\bigr\vert^2\Bigr)^{\frac{1}{2}}}_{L^p(\Omega)}\\
&\leq\, 
\Bignorm{\Bigl(\sum_{k,n\geq 1} k \bigl\vert T^{k-1}(I-T)x_n\bigr\vert^2\Bigr)^{\frac{1}{2}}}_{L^p(\Omega)}.
\end{align*}
Using  (\ref{5Equiv}) we deduce that
$$
\Bignorm{\sum_{n\geq 1} \varepsilon_n \otimes T^nx_n}_{{\rm Rad}(L^p(\Omega))}\,
\lesssim\,
\Bignorm{\sum_{n\geq 1} \varepsilon_n \otimes x_n}_{{\rm Rad}(L^p(\Omega))}.
$$
This shows the $R$-boundedness of $\{T^n\,:\, n\geq 1\}$.

Likewise using (\ref{5Equiv}) and (\ref{52Ave}), we have
\begin{align*}
\Bignorm{\sum_{n\geq 1} \varepsilon_n \otimes n T^{n-1}(I-T) x_n}_{{\rm Rad}(L^p(\Omega))}\,
& \lesssim\,
\Bignorm{\Bigl(\sum_{k,n\geq 1} k\bigl\vert n\, T^{k+n-2}(I-T)^2 x_n\bigr\vert^2\Bigr)^{\frac{1}{2}}}_{L^p(\Omega)}\\
& \lesssim\,
\Bignorm{\Bigl(\sum_{k,n\geq 1} (k+n)^3\bigl\vert T^{k+n-2}(I-T)^2 x_n\bigr\vert^2\Bigr)^{\frac{1}{2}}}_{L^p(\Omega)}\\
& \lesssim\,
\Bignorm{\Bigl(\sum_{n\geq 1}\sum_{k\geq n} (k+1)^3\bigl\vert T^{k-1}(I-T)^2 x_n\bigr\vert^2\Bigr)^{\frac{1}{2}}}_{L^p(\Omega)}\\
& \lesssim\,
\Bignorm{\Bigl(\sum_{k,n\geq 1} k^3\bigl\vert T^{k-1}(I-T)^2 x_n\bigr\vert^2\Bigr)^{\frac{1}{2}}}_{L^p(\Omega)}\\
& \lesssim\,
\Bignorm{\sum_{n\geq 1} \varepsilon_n \otimes x_n}_{{\rm Rad}(L^p(\Omega))}.
\end{align*}
Thus the set $\{n T^{n-1}(I-T)\, :\, n\geq 1\}$ is $R$-bounded as well, which completes the proof.
\end{proof}

\begin{remark}\label{5Alpha}
It is easy to check that the above proof and hence 
Theorem \ref{5Automatic} extend to the case when $L^p(\Omega)$ is replaced by a
reflexive Banach space with property $(\alpha)$. In particular this holds true on
any reflexive Banach lattice with finite cotype.
However we do not know whether Theorem \ref{5Automatic} holds true on {\it noncommutative} $L^p$-spaces.
\end{remark}

The above proof can be adapted to the sectorial case, which yields a 
slight improvement of the main result of \cite{CDMY}. We will
explain this point in a separate note \cite{LM4}.

\medskip
\section{From $H^\infty$ functional calculus to square functions}
The main aim of this section is to determine when a
Ritt operator $T$ with a bounded $H^{\infty}_{0}(B_\gamma)$ 
functional calculus necessarily satisfies square function estimates
$\norm{x}_{T,m}\lesssim\norm{x}$. We will show that this holds 
true on Banach spaces with a finite cotype. We refer the reader e.g. to 
\cite{DJT} for information on cotype.

For that purpose, we investigate a strong form 
of bounded holomorphic functional calculus which is somehow natural in order to make 
connections with square functions. We consider both the sectorial case 
and the Ritt case.

Let $f_1,\ldots,f_n$ be a finite family of $H^\infty(\O)$, for some non empty open set $\O\subset \Cdb$. 
In the sequel we let
$$
\Bignorm{\Bigl(\sum_{l=1}^{n}\vert f_l\vert^2\Bigr)^{\frac{1}{2}}}_{\infty,\footnotesize{\O}}\,
=\,\sup\biggl\{\Bigl(\sum_{l=1}^{n}\vert f_l(z)\vert^2\Bigr)^{\frac{1}{2}}\, :\, z\in \O\biggr\}.
$$
Equivalently, let $(e_1,\ldots,e_n)$ be the canonical basis of the Hermitian space $\ell^2_n$, then
\begin{equation}\label{3Square}
\Bignorm{\Bigl(\sum_{l=1}^{n}\vert f_l\vert^2\Bigr)^{\frac{1}{2}}}_{\infty, \footnotesize{\O}}\,
=\,\Bignorm{\sum_{l=1}^{n} f_l\otimes e_l}_{H^{\infty}(\footnotesize{\O};\ell^2_n)}.
\end{equation}

In the following definitions, $X$ is an arbitrary Banach space.

\begin{definition}\label{3Quad}
\
\begin{enumerate}
\item [(1)] Let $A$ be a sectorial operator of type $\omega\in (0,\pi)$ on $X$, 
and let $\theta\in (\omega,\pi)$. We say that $A$ admits a quadratic $H^{\infty}(\Sigma_\theta)$ 
functional calculus if there exists a constant $K>0$ such that for any $n\geq 1$, for any
$f_1,\ldots,f_n$ in $H^\infty_0(\Sigma_\theta)$, and for any $x\in X$,
\begin{equation}\label{3Q0}
\Bignorm{\sum_{l=1}^{n}\varepsilon_l\otimes f_l(A)x}_{{\rm Rad}(X)}\,\leq\, K\norm{x}\,
\Bignorm{\Bigl(\sum_{l=1}^{n}\vert f_l\vert^2\Bigr)^{\frac{1}{2}}}_{\infty,\Sigma_\theta}.
\end{equation}
\item [(2)] Let $T$ be a Ritt operator of type $\alpha\in \bigl(0,\frac{\pi}{2}\bigr)$ on $X$, 
and let $\gamma\in \bigl(\alpha,\frac{\pi}{2}\bigr)$. We say that $T$ admits a quadratic 
$H^{\infty}(B_\gamma)$ 
functional calculus if there exists a constant $K>0$ such that for any $n\geq 1$, for any
$\varphi_1,\ldots,\varphi_n$ in $H^{\infty}_{0}(B_\gamma)$, and for any $x\in X$,
$$
\Bignorm{\sum_{l=1}^{n}\varepsilon_l\otimes \varphi_l(T)x}_{{\rm Rad}(X)}\,\leq\, K\norm{x}\,
\Bignorm{\Bigl(\sum_{l=1}^{n}\vert \varphi_l\vert^2\Bigr)^{\frac{1}{2}}}_{\infty,B_\gamma}.
$$
\end{enumerate}
\end{definition}

Arguing as in Proposition \ref{2Approx}, one can restrict to polynomials in Part (2).

It is clear that any sectorial operator with  a quadratic $H^{\infty}(\Sigma_\theta)$ 
functional calculus has a bounded $H^{\infty}(\Sigma_\theta)$ 
functional calculus. We will see in Proposition \ref{3Cex} that the converse does not hold true. 
We are going to show however 
that up to a change of angle,
that converse holds true on a large class
of Banach spaces. We will need the following remarkable estimate
of Kaiser-Weis \cite[Cor. 3.4]{KaiWei}.

\begin{lemma}\label{3Kaiser} \cite{KaiWei}
Let $X$ be a Banach space with finite cotype. Then there exists a constant $C>0$ such that
\begin{equation}\label{3DefQ}
\Bignorm{\sum_{k,l\geq 1} \alpha_{kl}\,\varepsilon_k\otimes\varepsilon_l\otimes x_k}_{{\rm Rad}({\rm Rad}(X))}\,
\leq\, C\sup_{k}\Bigl(\sum_{l}\vert \alpha_{kl}\vert^2\Bigr)^{\frac{1}{2}}\,\Bignorm{\sum_{k}\varepsilon_k\otimes
x_k}_{{\rm Rad}(X)}
\end{equation}
for any finite family $(\alpha_{kl})_{k,l\geq 1}$ of complex numbers and
any finite family $(x_k)_{\geq 1}$ of $X$.
\end{lemma}

\begin{theorem}\label{3Q}
Assume that $X$ has finite cotype and let $A$ be a sectorial operator
on $X$ with a bounded $H^{\infty}(\Sigma_\theta)$ functional calculus. Then 
$A$ admits a quadratic $H^{\infty}(\Sigma_\nu)$ functional calculus for any 
$\nu\in(\theta,\pi)$.
\end{theorem}

\begin{proof}
The proof relies on a decomposition principle for holomorphic functions,
due to E. Franks and A. McIntosh. 
Let $0<\theta<\nu<\pi$ be two angles. That decomposition principle says 
that there exists a constant $C>0$, and two sequences $(F_k)_{k\geq 1}$ and
$(G_k)_{k\geq 1}$ in $H_{0}^{\infty}(\Sigma_{\theta})$ such that:
\begin{enumerate}
\item [(a)] For any $z\in\Sigma_{\theta}$, we have 
$\sum_{k\geq 1} \vert F_k(z)\vert\,\leq\, C$.

\smallskip
\item [(b)]
For any $z\in\Sigma_{\theta}$, we have 
$\sum_{k\geq 1} \vert G_k(z)\vert\,\leq\, C$.

\smallskip
\item [(c)] For any Banach space $Z$ and for any function $F\in H^{\infty}(\Sigma_\nu;Z)$,
there exists a bounded sequence $(b_k)_{k\geq 1}$ in $Z$ such that 
$$
\norm{b_k}\leq C\norm{F}_{H^{\infty}(\Sigma_{\nu};Z)},\qquad k\geq 1,
$$
and
$$
F(z)\,=\,\sum_{k=1}^{\infty} b_k\, F_k(z)G_k(z),\qquad z\in \Sigma_\theta.
$$
\end{enumerate}
Indeed, \cite[Prop. 3.1]{FMI} and the last paragraph of \cite[Section 3]{FMI}
show this property for $Z=\Cdb$. However it is easy to check that the proof works as well
for $Z$-valued holomorphic functions.

Since $A$ admits a bounded $H^{\infty}(\Sigma_\theta)$ functional calculus, 
we have a uniform estimate
$$
\Bignorm{\sum_k \eta_k F_k(A)}\,\lesssim\,\sup_{z\in\Sigma_\theta}\Bigl\vert 
\sum_k \eta_k F_k(z)\Bigr\vert\,\leq\, \sup_k\vert\eta_k\vert\,\sup_{z\in\Sigma_\theta}\sum_{k} \vert F_k(z)\vert
$$
for finite families $(\eta_k)_{k\geq 1}$ of complex numbers. Hence by (a), we have 
\begin{equation}\label{3(i)}
\sup_{m\geq 1} \,\sup_{\eta_k=\pm 1} \Bignorm{\sum_{k=1}^{m} \eta_k F_k(A)}\,<\infty\,.
\end{equation}
Likewise, (b) implies that
\begin{equation}\label{3(ii)}
\sup_{m\geq 1} \,\sup_{\eta_k=\pm 1} \Bignorm{\sum_{k=1}^{m}\eta_k G_k(A)}\,<\infty\,.
\end{equation}

We will apply property (c) with $Z=\ell^2_n$ for arbitrary $n\geq 1$. 
Let $f_1,\ldots, f_n$ in $H^{\infty}_0(\Sigma_\nu)$, and consider
$$
F= \,\sum_{l=1}^{n} f_l\otimes e_l\ \in H^{\infty}\bigl(\Sigma_\nu;\ell^2_n\bigr).
$$
Let $(b_k)_{k\geq 1}$ be the bounded sequence of $\ell^2_n$ provided by (c),
and write $b_k=(\alpha_{k1},\alpha_{k2},\ldots,\alpha_{kn})$ for any $k\geq 1$. 
Then
\begin{equation}\label{3Decomp}
f_l(z)=\,\sum_{k=1}^{\infty} \alpha_{kl}\, F_k(z)G_k(z),\qquad z\in \Sigma_\theta ,
\end{equation}
for any $l=1,\ldots, n$, and
\begin{equation}\label{3Estim}
\sup_{k}\Bigl(\sum_{l}\vert \alpha_{kl}\vert^2\Bigr)^{\frac{1}{2}}\, \leq C\,
\Bignorm{\Bigl(\sum_{l=1}^{n}\vert f_l\vert^2\Bigr)^{\frac{1}{2}}}_{\infty,\Sigma_\nu}
\end{equation}
by (c) and (\ref{3Square}).
 
For any $l=1,\ldots,n$ and any integer $m\geq 1$, we consider the function
$$
h_{m,l}=\sum_{k=1}^{m}\alpha_{kl}\, F_k G_k,
$$
which belongs to $H_{0}^{\infty}(\Sigma_\nu)$ and approximates $f_l$ by (\ref{3Decomp}). 

Let $x\in X$. By the Khintchine-Kahane inequality (see e.g. \cite[Thm. 1.e.13]{LT2}), we have 
\begin{align*}
\Bignorm{\sum_l \varepsilon_l\otimes h_{m,l}(A)x}_{{\rm Rad}(X)}\,
& =\,\biggl(\int_{\footnotesize{\M}} \Bignorm{\sum_{k,l}\varepsilon_l(u)
\alpha_{kl}F_k(A)G_k(A)x}^2\, d\Pdb(u) \,\biggr)^{\frac{1}{2}}\\
& \lesssim\, \int_{\footnotesize{\M}} \Bignorm{\sum_{k} F_k(A)\Bigl(\sum_l
\varepsilon_l(u)\alpha_{kl} G_k(A)x\Bigr)}\, d\Pdb(u) \,.
\end{align*}
For any $x_1,\ldots, x_m$ in $X$, we have
$$
\sum_{k} F_k(A) x_k\, =\,\int_{\footnotesize{\M}}\Bigl(\sum_k\varepsilon_k(v) F_k(A)
\Bigr)\Bigl(\sum_k \varepsilon_k(v) x_k\Bigr)\,
d\Pdb(v)\,,
$$
hence
\begin{align*}
\Bignorm{\sum_{k} F_k(A) x_k}\,
&\leq  \,\int_{\footnotesize{\M}} \Bignorm{\sum_k \varepsilon_k(v) F_k(A)}\Bignorm{\sum_k
\varepsilon_k(v) x_k}\,
d\Pdb(v)\\
&\lesssim \,\int_{\footnotesize{\M}} \Bignorm{\sum_k\varepsilon_k(v)  x_k}\,
d\Pdb(v)
\end{align*}
by (\ref{3(i)}). Applying this estimate with $x_k=\sum_l
\varepsilon_l(u)\alpha_{kl} G_k(A)x$ and integrating over $(u,v)\in\M\times\M$, we deduce that
$$
\Bignorm{\sum_l \varepsilon_l\otimes h_{m,l}(A)x}_{{\rm Rad}(X)}\,
\lesssim\,\Bignorm{\sum_{k,l}\alpha_{kl}\,\varepsilon_k\otimes\varepsilon_l\otimes G_k(A)x}_{{\rm Rad}({\rm Rad}(X))}.
$$
By assumption, $X$ has finite cotype. Hence it follows from Lemma \ref{3Kaiser} and
(\ref{3Estim}) that
$$
\Bignorm{\sum_l \varepsilon_l\otimes h_{m,l}(A)x}_{{\rm Rad}(X)}\,
\lesssim\,\Bignorm{\Bigl(\sum_{l}\vert f_l\vert^2\Bigr)^{\frac{1}{2}}}_{\infty,\Sigma_\nu}\,
\Bignorm{\sum_l \varepsilon_k\otimes G_k(A)x}_{{\rm Rad}(X)}.
$$
Moreover according to (\ref{3(ii)}), we have
$\bignorm{\sum_k \varepsilon_k\otimes G_k(A)x}_{{\rm Rad}(X)}\lesssim\norm{x}$. Thus 
we finally obtain
$$
\Bignorm{\sum_l \varepsilon_l\otimes h_{m,l}(A)x}_{{\rm Rad}(X)}\,
\lesssim\,\norm{x}\,
\Bignorm{\Bigl(\sum_{l}\vert f_l\vert^2\Bigr)^{\frac{1}{2}}}_{\infty,\Sigma_\nu}.
$$
We deduce the expected result by an entirely classical approximation process, that we 
explain for the convenience of the reader. For any $\varepsilon\in (0,1)$,
set $A_\varepsilon =(\varepsilon I+A)(I+\varepsilon A)^{-1}$. Then $A_\varepsilon$
is bounded and invertible, its spectrum is a compact subset of $\Sigma_\theta$,
and it follows from Cauchy's Theorem that for some contour $\Gamma_\varepsilon$ of finite length included in
the open set $\Sigma_\theta$, we have
$$
h(A_\varepsilon)=\,\frac{1}{2\pi i}\int_{\Gamma_\varepsilon} h(z) R(z,A_\varepsilon)\, dz
$$
for any $h\in H^{\infty}_{0}(\Sigma_\theta)$. 
Since $h_{m,l}\to f_l$ pointwise and $\sup_{m,z}\vert h_{m,l}(z)\vert\,<\infty\,$, the 
above integral representation ensures that 
$$
\lim_{m\to\infty}h_{m,l}(A_\varepsilon)\,=f_l(A_\varepsilon),\qquad l=1,\ldots,n.
$$
Furthermore,
$$
\lim_{\varepsilon\to 0}f_l(A_\varepsilon)\,=\,f_l(A)
$$
for any $l=1,\ldots,n$, by \cite[Lem. 2.4]{LM1}. 

Now observe that the $A_\varepsilon$'s uniformly admit a 
bounded $H^{\infty}(\Sigma_\theta)$ functional calculus, that is, there
exists a constant $K>0$ such that $\norm{h(A_\varepsilon)}\leq K\norm{h}_{\infty,\Sigma_\theta}$
for any $h\in H^{\infty}_{0}(\Sigma_\theta)$ and any $\varepsilon\in (0,1)$. It therefore 
follows from the above proof that there is a constant $K'>0$ such that 
\begin{equation}\label{3Approx}
\Bignorm{\sum_l \varepsilon_l\otimes h_{m,l}(A_\varepsilon)x}_{{\rm Rad}(X)}\,
\leq\, K'\,\norm{x}\,
\Bignorm{\Bigl(\sum_{l}\vert f_l\vert^2\Bigr)^{\frac{1}{2}}}_{\infty,\Sigma_\nu}
\end{equation}
for any $m\geq 1$ and any $\varepsilon\in (0,1)$. 
Then (\ref{3Q0}) follows from (\ref{3Approx}).
\end{proof}

We now state a similar result for Ritt operators and their functional calculus.

\begin{theorem}\label{3QQ}
Assume that $X$ has finite cotype and let $T$ be a Ritt operator
on $X$ with a bounded $H^{\infty}(B_\gamma)$ functional calculus. Then 
$T$ admits a quadratic $H^{\infty}(B_\nu)$ functional calculus for any 
$\nu\in\bigl(\gamma,\frac{\pi}{2}\bigr)$.
\end{theorem}

\begin{proof}
There are two ways to get to this result. The first one is to
mimic the proof of Theorem \ref{3Q}, using a Franks-McIntosh decomposition
adapted to Stolz domains. The existence of such decompositions follows 
from \cite[Section 5]{FMI}. 

The second way is to 
observe that the transfer principle stated as Proposition \ref{4TA}
holds true (with essentially the same proof) for the quadratic functional
calculus. Namely with $A=I-T$, the following are equivalent:
\begin{enumerate}
\item [(i)] $T$ admits a quadratic $H^{\infty}(B_\gamma)$ functional
calculus for some $\gamma\in \bigl(0,\frac{\pi}{2}\bigr)$.
\item [(ii)] $A$ admits a quadratic $H^{\infty}(\Sigma_\theta)$ functional
calculus for some $\theta\in \bigl(0,\frac{\pi}{2}\bigr)$.
\end{enumerate}
Hence the result follows from Theorem \ref{3Q}, the implication
`(i)$\Rightarrow$(ii)' of Proposition \ref{4TA} and the implication 
`(ii)$\Rightarrow$(i)' above. 
\end{proof}

Banach spaces with property $(\alpha)$ have finite cotype, hence they 
satisfy Theorems \ref{3Q} and 
\ref{3QQ}. It turns out that a much stronger $H^\infty$ calculus 
property holds on those spaces, as follows.

\begin{proposition}\label{3Alpha2}
Assume that $X$ has property $(\alpha)$.
\begin{enumerate}
\item [(1)] Let $A$ be a sectotial operator
on $X$ with a bounded $H^{\infty}(\Sigma_\theta)$ functional calculus. Then 
for any 
$\nu\in(\theta,\pi)$, there exists a constant $K>0$ such that
\begin{equation}\label{3Mat1}
\Bignorm{\sum_{l,j=1}^{n}\varepsilon_l\otimes f_{lj}(A)x_j}_{{\rm Rad}(X)}\,\leq
K\sup_{z\in\Sigma_\nu}\bignorm{[f_{lj}(z)]}_{M_n}\,
\Bignorm{\sum_{j=1}^{n}\varepsilon_j\otimes  x_j}_{{\rm Rad}(X)}
\end{equation}
for any $n\geq 1$, for any matrix $[f_{lj}]$ of elements of $H_{0}^{\infty}(\Sigma_\nu)$ and for any 
$x_1,\ldots, x_n$ in $X$.
\item [(2)] Let $T$ be a Ritt operator
on $X$ with a bounded $H^{\infty}(B_\gamma)$ functional calculus. Then 
for any  $\nu\in\bigl(\gamma,\frac{\pi}{2}\bigr)$, there exists a constant $K>0$ such that
\begin{equation}\label{3Mat2}
\Bignorm{\sum_{l,j=1}^{n}\varepsilon_l\otimes \varphi_{lj}(T)x_j}_{{\rm Rad}(X)}\,\leq
K\sup_{z\in B_\nu}\bignorm{[\varphi_{lj}(z)]}_{M_n}\,
\Bignorm{\sum_{j=1}^{n}\varepsilon_j\otimes  x_j}_{{\rm Rad}(X)}
\end{equation}
for any $n\geq 1$, for any matrix $[\varphi_{lj}]$ of elements of $H_{0}^{\infty}(B_\nu)$ and for any 
$x_1,\ldots, x_n$ in $X$.
\end{enumerate}
\end{proposition}

\begin{proof} A Banach space $X$ with property $(\alpha)$
satisfies the following property: there exists a constant $C>0$ such that for
any $n\geq 1$, for any finite family $(b_k)_{k\geq 1}$ of elements of $M_n$
that we denote by $b_k=[b_k(l,j)]_{1\leq l,j\leq n}$ and for any $n$-tuple
$(x_{k1})_{k\geq 1},\ldots,(x_{kn})_{k\geq 1}$ of families in $X$,
\begin{equation}\label{3Strong-Alpha}
\Bignorm{\sum_{k\geq 1}\sum_{l,j=1}^{n} \varepsilon_k\otimes 
\varepsilon_l\otimes b_k(l,j) x_{kj}}_{{\rm Rad}({\rm Rad}(X))}\, \leq
\, C\sup_k\norm{b_k}_{M_n}\,\Bignorm{\sum_{k,j}\varepsilon_k\otimes
\varepsilon_j\otimes x_{kj}}_{{\rm Rad}({\rm Rad}(X))}.
\end{equation}
This strengthening of (\ref{3DefQ}) for those spaces 
is due to Haak and Kunstmann \cite[Lem. 5.2]{HK} (see also \cite{H}). 

Let us explain (1). We consider an $n\times n$ matrix $[f_{lj}]$ of elements 
of $H^{\infty}_{0}(\Sigma_\nu)$ and we associate $F\in H^\infty(\Sigma_\nu;M_n)$
defined by
$$
F(z) =\bigl[f_{lj}(z)\bigr],\qquad z\in\Sigma_\nu.
$$
Then  arguing as in the proof of Theorem \ref{3Q}
and applying the Franks-McIntosh decomposition principle with $Z=M_n$, we find
a sequence $(b_k)_{k\geq 1}$ of $n\times n$ matrices $b_k=[b_k(l,j)]_{1\leq l,j\leq n}$
such that 
$$
f_{lj} (z) \,=\,\sum_{k=1}^{\infty} b_k(l,j)\, F_k(z)G_k(z),\qquad z\in\Sigma_\theta,
$$
for any $l,j=1,\ldots,n$, and
$$
\sup_k\norm{b_k}_{M_n}\,\leq\, C\,\sup\bigl\{\bignorm{\bigl[f_{lj}(z)\bigr]}_{M_n}\, :\, z\in \Sigma_\nu\bigr\}.
$$
Using the above results in the place of (\ref{3Decomp}) and (\ref{3Estim}), the estimate 
(\ref{3Strong-Alpha}) in the place of (\ref{3DefQ}), and arguing as in the proof
of Theorem \ref{3Q}, we obtain (\ref{3Mat1}). Details are left to the reader.

Part (2) can be deduced from part (1) in the same manner that Theorem \ref{3QQ}
was deduced from Theorem \ref{3Q}.
\end{proof}

Part (2) of the above proposition generalizes \cite[Thm. 3.3]{LMX1}, 
where this property is proved for (commutative) $L^p$-spaces.

\begin{remark} 
Property (\ref{3Mat1}) means that
the homomorphism $H^{\infty}_{0}(\Sigma_\nu)\to B(X)$ induced by the functional calculus
is matricially $R$-bounded in the sense of \cite[Section 4]{KL}. Restricting this property
to column matrices, we obtain the property proved in Theorem \ref{3Q}. On the other hand,
restricting (\ref{3Mat1}) to diagonal matrices, we find the following  
result of Kalton-Weis \cite{KaW2} (see also \cite[Thm. 12.8]{KW}): if $A$
has a bounded $H^{\infty}(\Sigma_\theta)$ functional calculus on $X$ with property $(\alpha)$,
then for any $\nu\in(\theta,\pi)$, the functional calculus
homomorphism $H^{\infty}_{0}(\Sigma_\nu)\to B(X)$ maps the unit ball of $H^{\infty}_{0}(\Sigma_\nu)$
into an $R$-bounded subset of $B(X)$.
\end{remark}

We are now going to show that Theorem \ref{3Q} does not hold true 
on all Banach spaces, namely the next proposition
shows that it fails on $c_0$. 
A similar construction shows that 
Theorem \ref{3QQ} also fails on $c_0$.

\begin{proposition}\label{3Cex} Let $A\colon c_0\to c_0$ be defined by 
$$
A(w) = \bigl(2^{-j}w_j\bigr)_{j\geq 1},\qquad w=(w_j)_{j\geq 1}\in c_0.
$$
Then $A$ is a sectorial operator and for any $\theta\in (0,\pi)$:
\begin{enumerate}
\item [(1)] $A$ admits a bounded $H^{\infty}(\Sigma_\theta)$
functional calculus.
\item [(2)] $A$ does not have a quadratic $H^{\infty}(\Sigma_\theta)$
functional calculus.
\end{enumerate}
\end{proposition}

\begin{proof}
The facts that $A$ is sectorial and that property (1) holds are easy. Indeed, for any 
$\theta\in(0,\pi)$ and any $f\in H^{\infty}_{0}(\Sigma_\theta)$, we have
$$
[f(A)](w)=(f(2^{-j})w_j)_{j\geq 1}
$$
for any $w=(w_j)_{j\geq 1}$ in $c_0$,
and hence 
$$
\norm{f(A)}\leq\norm{f}_{L^{\infty}(0,\infty)}.
$$

To prove (2), let us assume that $A$ admits a quadratic $H^{\infty}(\Sigma_\theta)$
functional calculus for some $\theta\in(0,\pi)$. Let $(e_j)_{j\geq 1}$ denote the
canonical basis of $c_0$. For any integers  $n,m\geq 1$, for any  $w_1,\ldots, w_m$ in $\Cdb$
and any $f_1,\ldots,f_n$ in $H^{\infty}_0(\Sigma_\theta)$, 
$$
\sum_l\varepsilon_l\otimes f_l(A)\Bigl(\sum_j w_j e_j\Bigr)\,=\,\sum_{l,j}
f_l(2^{-j}) w_j\, \varepsilon_l \otimes e_j\,.
$$
Hence there is a constant $C\geq 0$ (not depending on $n,m,w_j$) such  that
$$
\Bignorm{
\sum_{l=1}^n\sum_{j=1}^{m} 
f_l(2^{-j}) w_j\, \varepsilon_l \otimes e_j}_{{\rm Rad}(c_0)}\,\leq
C\sup_j\vert w_j\vert\,\Bignorm{\Bigm(\sum_{l=1}^{n}\vert f_l\vert^2\Bigr)^{\frac{1}{2}}}_{\infty, \Sigma_\theta}
$$
for any $f_1,\ldots,f_n$ in $H^{\infty}_0(\Sigma_\theta)$, 
By an entirely classical approximation argument, the above estimate holds as well
when the $f_l$'s belong to $H^{\infty}(\Sigma_\theta)$. Applying this with
$w_j=1$ for all $j$, one obtains 
\begin{equation}\label{3Cex1}
\Bignorm{\sum_{l=1}^n\sum_{j=1}^{m} 
f_l(2^{-j})\, \varepsilon_l \otimes e_j}_{{\rm Rad}(c_0)}\,\leq
C \,\Bignorm{\Bigm(\sum_{l=1}^{n}\vert f_l\vert^2\Bigr)^{\frac{1}{2}}}_{\infty, \Sigma_\theta},
\qquad f_1,\ldots,f_n\in H^{\infty}(\Sigma_\theta).
\end{equation}
Let $Q_{n,m}\colon H^{\infty}(\Sigma_\theta;\ell^2_n)\to \ell^\infty_{m}(\ell^2_n)$ be defined
by $Q_{n,m}(F)=\bigl(F(2^{-j})\bigr)_{1\leq j\leq m}$. Then $Q_n$ is onto and the vectorial 
form of Carleson's Interpolation Theorem (see \cite[VII.2]{Ga}) ensures that its lifting constant is bounded
by a universal constant not depending on either $m$ or $n$. Thus there is a constant $K\geq 1$
such that for any family $(\alpha_{lj})_{1\leq j\leq m,\, 1\leq l\leq n}$ of complex
numbers there exists $f_1,\ldots,f_n$ in $H^{\infty}(\Sigma_\theta)$ such that
$$
\Bignorm{\Bigm(\sum_{l=1}^{n}\vert f_l\vert^2\Bigr)^{\frac{1}{2}}}_{\infty, \Sigma_\theta}
\leq K
\sup_{1\leq j\leq m}\Bigl(\sum_{l=1}^{n}\vert\alpha_{lj}\vert^2\Bigr)^{\frac{1}{2}}
\qquad\hbox{and}\qquad 
\alpha_{lj}=f_l(2^{-j})
$$
for any $1\leq j\leq m,\, 1\leq l\leq n$. It therefore follows from (\ref{3Cex1})
that
\begin{equation}\label{3Cex2}
\Bignorm{\sum_{l=1}^n\sum_{j=1}^{m} 
\alpha_{lj}\, \varepsilon_l \otimes e_j}_{{\rm Rad}(c_0)}\,\leq
CK \,
\sup_{1\leq j\leq m}\Bigl(\sum_{l=1}^{n}\vert\alpha_{lj}\vert^2\Bigr)^{\frac{1}{2}},
\qquad \alpha_{lj}\in\Cdb.
\end{equation}
Let $n\geq 1$ be an integer. Since the unit ball of $\ell^2_n$ is compact,
there exists a finite family $(y_1,\ldots, y_m)$ of that unit ball such that
\begin{equation}\label{3Cex4}
\norm{y}_{\ell^2_n} \leq 2\sup\bigl\{\vert\langle y,y_j\rangle\vert\, :\, j=1,\ldots, m\bigr\}
\end{equation}
for any $y\in\ell^2_n$. Let $(h_1,\ldots,h_n)$ be an orthonormal basis of $\ell^2_n$, 
and let
$$
\alpha_{lj} =\langle h_l,y_j\rangle,\qquad 1\leq j\leq m,\, 1\leq l\leq n.
$$
Then the supremum in the right handside of (\ref{3Cex2}) is equal to 
$\sup_j\norm{y_j}$, hence is less than or equal to $1$. Consequently,
$$
\Bignorm{\sum_{l^=1}^n\sum_{j=1}^{m} 
\langle h_l,y_j\rangle\, \varepsilon_l \otimes e_j}_{{\rm Rad}(c_0)}\,\leq
CK.
$$
Now observe that for any $u\in\M$,
\begin{align*}
\Bignorm{\sum_{l=1}^n\sum_{j=1}^{m} 
\langle h_l,y_j\rangle\, \varepsilon_l(u)  e_j}_{c_0}
\, & = \,
\Bignorm{\sum_{j=1}^{m} \Bigl\langle 
\sum_{l=1}^n \varepsilon_l(u) \, h_l, y_j\Bigr\rangle\, e_j}_{c_0}
\\ &
=\,
\sup_{j}\Bigl\vert \Bigl\langle\sum_{l=1}^n \varepsilon_l(u) \, h_l, y_j\Bigr\rangle\Bigr\vert.
\end{align*}
Since $(h_1,\ldots, h_n)$ is an orthonormal basis, the norm of $\sum_{l^=1}^n \varepsilon_l(u) \, h_l$
in $\ell^2_n$ is equal to $n^{\frac{1}{2}}$. 
Aplyinh (\ref{3Cex4}), we deduce that
$$
n^{\frac{1}{2}}\leq 2\,\Bignorm{\sum_{l=1}^n\sum_{j=1}^{m} 
\langle h_l,y_j\rangle\, \varepsilon_l(u)  e_j}_{c_0}.
$$
Integrating over $\M$, this yields $n^{\frac{1}{2}}\leq 2 CK$ for any $n\geq 1$, a
contradiction.
\end{proof}

We now come back to the question addressed at the beginning of this section.
The following classical result will be used in the next proof:
If a Banach space $X$
does not contain $c_0$ (as an isomorphic subspace), 
then a series $\sum_k\varepsilon_k\otimes x_k\,$ converges in $L^2(\M;X)$
if and only if its partial sums are uniformly bounded (see \cite{Kwa}).

\begin{proposition}\label{4SFE} Assume that $X$ does not contain $c_0$.
Let $T\colon X\to X$ be a Ritt operator and assume that $T$ has a quadratic $H^{\infty}(B_\gamma)$ functional
calculus for some $\gamma\in \bigl(0,\frac{\pi}{2}\bigr)$. Then for any $m\geq 1$, 
$T$ satisfies a uniform estimate 
\begin{equation}\label{4SFE1}
\norm{x}_{T,m}\lesssim\norm{x},\qquad x\in X.
\end{equation}
\end{proposition}

\begin{proof}
According to the property discussed before the statement, it suffices to show
the existence of a constant $K>0$ such that for any $n\geq 1$ and any $x\in X$,
$$
\Bignorm{\sum_{l=1}^{n} l^{m-\frac{1}{2}}\varepsilon_l\otimes T^{l-1}(I-T)^m x}_{{\rm Rad}(X)}\leq
K\norm{x}.
$$
This is obtained by applying Definition \ref{3Quad}, (2), with 
$$
\varphi_l(z) = l^{m-\frac{1}{2}} z^{l-1}(1-z)^l,
$$
see the proof of \cite[Thm. 3.3]{LMX1} for the details.
\end{proof}

Let us finally summarize what we obtain by combining Theorem \ref{3QQ} 
and Proposition \ref{4SFE}. Recall that a Banach space with finite cotype
cannot contain $c_0$.

\begin{corollary}\label{4SFE2} Assume that $X$ has finite cotype. Let 
$T\colon X\to X$ be a Ritt operator with a bounded $H^{\infty}(B_\gamma)$ functional
calculus for some $\gamma\in \bigl(0,\frac{\pi}{2}\bigr)$. 
Then it satisfies a square function estimate (\ref{4SFE1}) for any $m\geq 1$.
\end{corollary}

\begin{remark}
It follows from the proof of Proposition \ref{3Cex} that the spaces
$\ell^\infty_n$ do not satisfy (\ref{3DefQ}) uniformly, that is, there
is no common constant $C>0$ such that (\ref{3DefQ}) holds with
$X=\ell^\infty_n$ for any $n\geq 1$. Moreover a Banach space with no finite cotype
contains the $\ell^\infty_n$'s uniformly as isomorhic subspaces (see e.g. \cite[Thm. 14.1]{DJT})
and hence cannot satisfy (\ref{3DefQ}).
Together with Lemma \ref{3Kaiser}, this observation shows that
a Banach space satisfies an estimate (\ref{3DefQ}) if and only if it has finite 
cotype.

In an early version of this paper, Theorem \ref{3Q} was stated under the assumption
that $X$ satisfies an estimate (\ref{3DefQ}). I had overlooked \cite[Cor. 3.4]{KaiWei} and
realized only recently that (\ref{3DefQ}) is the same as `finite cotype'. This 
led to the present neater presentation of Section 6.

Bernhard Haak and Markus Haase have informed me that they obtained a variant of Theorem \ref{3Q}
in a work in progress (see \cite{HH}). This work is independent of mine, and was 
undertaken several months ago. 
\end{remark}

\medskip
\section{From square functions to $H^\infty$ functional calculus}
This section is devoted to the issue of showing that a Ritt 
operator has a bounded $H^\infty$-functional calculus with respect to a Stolz domain
$B_\gamma$, provided that it satisfies suitable square function estimates.
We consider an arbitrary Banach space $X$ and first establish a 
general result, namely Theorem \ref{6FromSFE} below. Then we consider 
special cases in the last part of the section.

\begin{lemma}\label{6Nevanlinna}
Let $0<\alpha<\gamma< \frac{\pi}{2}$ and 
let $T\colon X\to X$ be a Ritt operator of type $\alpha$
(resp. an $R$-Ritt operator of $R$-type $\alpha$). There exists a 
constant $C>0$ such that for any $\varphi\in
H^{\infty}_0(B_\gamma)$, we have
$$
k\bignorm{\varphi(T)\bigl(T^{k} -T^{k-1}\bigr)}\,\leq\, C\norm{\varphi}_{\infty,
B_\gamma},\qquad k\geq 1
$$
(resp. the set $\{k\varphi(T)(T^{k} -T^{k-1})\, :\, k\geq 1\}$ is $R$-bounded
and
$$
\R\Bigl(\bigl\{k\varphi(T)\bigl(T^{k} -T^{k-1}\bigr)\, :\, k\geq 1\bigr\}\Bigr)\,
\leq\, C\norm{\varphi}_{\infty,
B_\gamma}\,\bigl).
$$
\end{lemma}

\begin{proof}
We will prove this result in the `$R$-Ritt case' only, the `Ritt case'
being similar and simpler. We fix a real number
$\beta\in(\alpha,\gamma)$. Recall Lemma \ref{5Blunck} and let 
$$
C_1=\R\Bigl(\bigl\{(\lambda-1)R(\lambda,T)\, :\, \lambda\in\partial B_\beta\setminus\{1\}
\bigr\}\Bigr).
$$
For any function $\varphi\in
H^{\infty}_0(B_\gamma)$ and any integer $k\geq 1$, we have
$$
k\varphi(T)\bigl(T^{k} -T^{k-1}\bigr)\,=\,\frac{1}{2\pi i}\,\int_{\partial B_\beta}
k\varphi(\lambda) \lambda^{k-1}\bigl((\lambda-1)R(\lambda,T)\bigr)\, d\lambda\,.
$$
Hence by Lemma \ref{5Convexity}, we have
\begin{align*}
\R\Bigl(\bigl\{k\varphi(T)\bigl(T^{k} -T^{k-1}\bigr)\, :\, k\geq 1\bigr\}\Bigr)\,
&\leq\, \frac{C_1}{\pi}\,\sup_{k\geq 1}\Bigl\{
k \int_{\partial B_\beta}
\vert \varphi(\lambda)\vert \,\vert\lambda\vert^{k-1}\,\vert d\lambda\vert
\Bigr\}\\
&\leq\, \frac{C_1}{\pi}\,\norm{\varphi}_{\infty,
B_\gamma}\,\sup_{k\geq 1}\Bigl\{
k \int_{\partial B_\beta}
\vert\lambda\vert^{k-1}\,\vert d\lambda\vert
\Bigr\}.
\end{align*}
The finiteness of the latter supremum is well-known, see e.g. \cite[Lem. 2.1]{V} and its proof.
The result follows at once.
\end{proof}

\begin{lemma}\label{6Decomp}
Let $T\colon X\to X$ be a Ritt operator. For any $x\in\overline{{\rm Ran}(I-T)}$, we have
$$
\sum_{k=1}^{\infty} k(k+1)T^{k-1}(I-T)^3  x\, =\, 2x.
$$
\end{lemma}

\begin{proof}
Let $N\geq 1$ be an integer. First, we have
\begin{align*}
\sum_{k=1}^{N} k(k+1)T^{k-1}(I-T)\, & =\,\sum_{k=1}^{N}k(k+1)T^{k-1}\,-\,\sum_{k=2}^{N+1}(k-1)k T^{k-1}\\ 
& =\, 
2\sum_{k=1}^{N}kT^{k-1}\, -N(N+1)T^N.
\end{align*}
Then we compute
$$
\sum_{k=1}^{N}kT^{k-1}(I-T)\, =\, \sum_{k=1}^{N}kT^{k-1} \,-\,\sum_{k=2}^{N+1}(k-1)T^{k-1}\, =\,
\sum_{k=1}^{N} T^{k-1} \, -NT^N, 
$$
and we note that
$$
\sum_{k=1}^{N}T^{k-1}(I-T)\, = I-T^N.
$$
Putting these identities together, we obtain that
\begin{equation}\label{6Identity}
\sum_{k=1}^{N} k(k+1)T^{k-1}(I-T)^3\,=\, 2I - 2 T^N -2NT^N(I-T) - N(N+1)T^N(I-T)^2.
\end{equation}
Since $T$ is a Ritt operator, the four sequences
$$
\S_0 =(T^N)_{N\geq 1},\quad \S_1=\bigl(NT^N(I-T)\bigr)_{N\geq 1}, \quad
\S_2=\bigl(N^2T^N(I-T)^2\bigr)_{N\geq 1}
$$
and
$$
\S_3=\bigl(N^3T^N(I-T)^3\bigr)_{N\geq 1}
$$
are bounded (see \cite[Lem. 2.1]{V}).

If $x=(I-T)z$ is an element of ${\rm Ran}(I-T)$, the boundedness of $\S_3$ implies that
$$
N(N+1)T^N(I-T)^2x\,=\,\frac{N+1}{N^2}\,N^3T^N(I-T)^3z\,\longrightarrow\, 0
$$
when $N\to \infty$. Then the boundedness of the sequence $\S_2$ 
implies that we actually have
$$
\lim_N N(N+1)T^N(I-T)^2x\, =\,0
$$
for any 
$x$ in the closure $\overline{{\rm Ran}(I-T)}$. Likewise, using $\S_2,\S_1$ and $\S_0$, we have
$$
\lim_N NT^N(I-T) x\,=\,0\qquad\hbox{and}\qquad \lim_N T^Nx\,=\,0
$$
for any 
$x\in\overline{{\rm Ran}(I-T)}$. Thus applying (\ref{6Identity}) yields the result.
\end{proof}

\begin{theorem}\label{6FromSFE}
Let $T\colon X\to X$ be an $R$-Ritt operator of $R$-type $\alpha\in
\bigl(0,\frac{\pi}{2}\bigr)$. If
$T$ and $T^*$
both  satisfy uniform estimates
$$
\norm{x}_{T,1}\,\lesssim\,\norm{x}\qquad\hbox{and}\qquad
\norm{y}_{T^*,1}\,\lesssim\,\norm{y}
$$
for $x\in X$ and $y\in X^*$, then $T$ admits a 
bounded $H^{\infty}(B_\gamma)$ functional calculus 
for any $\gamma\in\bigl(\alpha,\frac{\pi}{2}\bigr)$.
\end{theorem}

\begin{proof}
We fix $\gamma$ in $\bigl(\alpha,\frac{\pi}{2}\bigr)$. 
Let $\omega=e^{\frac{2i\pi}{3}}$, then 
the two operators $\omega I-T$ and $\bar{\omega} I-T$ are invertible and 
$I-T^3=(I-T)(\omega I-T)(\bar{\omega} I -T)$. Hence
$$
{\rm Ran}(I-T)\,=\,{\rm Ran}(I-T^3).
$$
Note moreover that $T^3$ is a Ritt operator.

Let $\varphi\in \P$ such that $\varphi(1)=0$ and consider $x\in X$.
Then $\varphi(T)x\in{\rm Ran}(I-T)$ hence applying
Lemma \ref{6Decomp} to $T^3$ and the above observations, we obtain
$$
\sum_{k=1}^{\infty} k(k+1)T^{3(k-1)}(I-T^3)^3\varphi(T)x\,=\, 2\varphi(T)x.
$$
We set $\psi(T)=(I+T+T^2)^3/2$ for convenience, so that $2\psi(T)(I-T)^3=(I-T^3)^3$. 
Then for any $y\in X^*$, we derive
\begin{align*}
\bigl\langle\varphi(T)x,y\bigr\rangle\, 
& = 
\,\sum_{k=1}^{\infty} \bigl\langle k(k+1)\psi(T)\varphi(T) T^{3(k-1)}(I-T)^3 x,y\bigr\rangle\\
& = 
\,\sum_{k=1}^{\infty} \bigl\langle \bigl[(k+1)\varphi(T) T^{k-1}(I-T)\bigr]
k^{\frac{1}{2}}T^{k-1}(I-T) x, k^{\frac{1}{2}}T^{*(k-1)}(I-T^*)\psi(T^*) y\bigr\rangle\,.
\end{align*}

Note that for any finite families $(x_k)_{k\geq 1}$ in $X$ and $(y_k)_{k\geq 1}$ in $X^*$, we have
$$
\sum_k \langle x_k,y_k\rangle\,=\,\int_{\footnotesize{\M}}\Bigl\langle\,
\sum_k\varepsilon_k(u) x_k,\sum_k\varepsilon_k(u) y_k\Bigr\rangle\, d\Pdb(u),
$$
and hence
$$
\Bigl\vert\sum_k \langle x_k,y_k\rangle\,\Bigr\vert
\,\leq\,\Bignorm{\sum_k \varepsilon_k\otimes x_k}_{{\rm Rad}(X)}\,
\Bignorm{\sum_k \varepsilon_k\otimes y_k}_{{\rm Rad}(X^*)}
$$
by Cauchy-Schwarz.

Thus for any integer $N\geq 1$, we have 
\begin{align*}
\biggl\vert  \sum_{k=1}^{N} & \bigl\langle \bigl[(k+1)\varphi(T) 
T^{k-1}(I-T)\bigr]
k^{\frac{1}{2}}T^{k-1}(I-T) x, k^{\frac{1}{2}}T^{*(k-1)}(I-T^*)\psi(T^*) y\bigr\rangle\,\biggr\vert\\
& \leq\,\Bignorm{\sum_{k=1}^{N}\varepsilon_k\otimes \bigl[(k+1)\varphi(T) 
T^{k-1}(I-T)\bigr] 
k^{\frac{1}{2}}T^{k-1}(I-T) x}_{{\rm Rad}(X)}\\ & \ \,\quad \qquad\qquad\qquad\times 
\Bignorm{\sum_{k=1}^{N}\varepsilon_k\otimes 
k^{\frac{1}{2}}T^{*(k-1)}(I-T^*)\psi(T^*) y}_{{\rm Rad}(X^*)}\\
& \leq\,\R\Bigl(\bigl\{(k+1)\varphi(T)T^{k-1}(I-T)\, :\, k\geq 1\bigr\}\Bigr)
\norm{\psi(T)}\norm{x}_{T,1}\norm{y}_{T^*,1}\\
& \lesssim\,\norm{\varphi}_{\infty, B_\gamma}\norm{x}_{T,1}\norm{\psi(T^*) y}_{T^*,1}
\end{align*}
by Lemma \ref{6Nevanlinna}.
Applying our assumptions, we deduce that 
$$
\bigl\vert\langle\varphi(T)x,y\bigr\rangle\bigr\vert\, \lesssim 
\norm{\varphi}_{\infty,B_\gamma} \,\norm{x}\norm{y}.
$$
Since $x$ and $y$ are arbitrary, this implies an estimate $\norm{\varphi(T)}\lesssim 
\norm{\varphi}_{\infty, B_\gamma}$ for polynomials vanishing at $1$. Writing any polynomial
as $\varphi = \varphi(1) +(\varphi-\varphi(1))$, we immediatly derive a similar 
estimate for all polynomials. 
This yields the result by Proposition \ref{2Approx}.
\end{proof}

Theorem 
\ref{6FromSFE} fails if we remove one of the two 
square function estimates in the assumption. This will 
follow from Proposition \ref{7One}.

We finally consider a special case and combinations with results
from the previous sections. Following \cite{KaW1}, we say that 
a Banach space $X$ has property $(\Delta)$ if the triangular projection
is bounded on ${\rm Rad}(
{\rm Rad}(X))$, that is, there exists 
a constant $C>0$ such that for finite doubly indexed families $(x_{kl})_{k,l\geq 1}$ of $X$,
$$
\Bignorm{\sum_{k\geq 1}\sum_{l\geq k}\varepsilon_k\otimes\varepsilon_l\otimes x_{kl}}_{{\rm Rad}(
{\rm Rad}(X))}\,\leq C\,
\Bignorm{\sum_{k\geq 1}\sum_{l\geq 1}\varepsilon_k\otimes\varepsilon_l\otimes x_{kl}}_{{\rm Rad}(
{\rm Rad}(X))}.
$$
That condition is clearly
weaker than $(\alpha)$. Furthermore
any UMD Banach space has property $(\Delta)$, by \cite[Prop. 3.2]{KaW1}. Thus any noncommutative $L^p$-space
with $1<p<\infty$ has property $(\Delta)$. On the other hand, property $(\Delta)$ does not hold uniformly 
on the spaces $\ell^\infty_n$, hence any Banach space with property $(\Delta)$ has finite cotype.

It follows from \cite[Thm. 5.3]{KaW1} that if $A$
is a sectorial operator on $X$ with property $(\Delta)$ and $A$ admits
a bounded $H^{\infty}(\Sigma_\theta)$ for some $\theta<\frac{\pi}{2}$, then $A$
is $R$-sectorial of $R$-type $<\frac{\pi}{2}$. Combining with (the easy implication
of) Proposition \ref{4TA} and Lemma \ref{5Blunck}, we deduce the 
following.

\begin{proposition}\label{6Delta}
Let $T$ be a Ritt operator on $X$ with property $(\Delta)$. If $T$ admits
a bounded $H^{\infty}(B_\gamma)$ for some $\gamma<\frac{\pi}{2}$, then $T$
is $R$-Ritt.
\end{proposition}

Combining further with Corollary \ref{4SFE2}, we obtain the following equivalence 
result.

\begin{corollary}\label{6Equiv} \
\begin{enumerate}
\item [(1)] 
Assume that $X$ has property $(\Delta)$ and let
$T\colon X\to X$ be a Ritt operator. The following assertions are equivalent.
\begin{enumerate}
\item [(i)] $T$ admits a 
bounded $H^{\infty}(B_\gamma)$ functional calculus 
for some $\gamma\in\bigl(0,\frac{\pi}{2}\bigr)$.
\item [(ii)] $T$ is $R$-Ritt and 
$T$ and $T^*$
both  satisfy uniform estimates
$$
\norm{x}_{T,1}\,\lesssim\,\norm{x}\qquad\hbox{and}\qquad
\norm{y}_{T^*,1}\,\lesssim\,\norm{y}
$$
for $x\in X$ and $y\in X^*$
\end{enumerate}
\item [(2)] Part (1) applies to Banach spaces with property $(\alpha)$ and to noncommutative
$L^p$-spaces for $1< p<\infty$.
\end{enumerate}
\end{corollary}

If $X=L^p(\Omega)$ is a commutative $L^p$-space with $1<p<\infty$, then 
demanding that $T$ is $R$-Ritt in condition (ii) is superfluous, by 
Theorem \ref{5Automatic}. In this case, the above statement yields
Theorem \ref{1MainLp}. According to this discussion and Remark 
\ref{5Alpha}, Theorem \ref{1MainLp} holds true as
well on any reflexive space with property $(\alpha)$.

We conclude this section with an observation of independent interest
on the role of the $R$-Ritt condition in the study of 
$H^\infty(B_\gamma)$ functional calculus. Recall Definition \ref{2Pb}.

\begin{proposition}\label{6Reduce}
Let $T\colon X\to X$ be an $R$-Ritt operator of $R$-type $\alpha$.
If $T$ is polynomially bounded, then it admits a bounded
$H^\infty(B_\gamma)$ functional calculus for any $\gamma\in \bigl(\alpha,\frac{\pi}{2}\bigr)$.
\end{proposition}

\begin{proof}
As was observed in Section 5, the 
operator $A=I-T$ is $R$-sectorial of $R$-type $\alpha$. 
Moreover the proof of the easy implication `(i)$\Rightarrow$(ii)'
of Proposition \ref{4TA} shows that $A$ admits a bounded
$H^\infty(\Sigma_{\frac{\pi}{2}})$ functional calculus. 
According to \cite[Prop. 5.1]{KaW1}, this implies that for any
$\theta\in (\alpha,\pi)$, $A$ admits a  bounded
$H^\infty(\Sigma_{\theta})$ functional calculus. 
The result therefore follows from Proposition \ref{4TA}.
\end{proof}

The $R$-boundedness assumption is essential in the above result.
Indeed with F. Lancien we show in \cite{LL} the existence of
Ritt operators which are polynomially bounded without admitting
any bounded
$H^\infty(B_\gamma)$ functional calculus.

\medskip
\section{Examples and illustrations}
In this final section, we give complements for the following 3 
classes of Banach spaces: Hilbert spaces, commutative $L^p$-spaces,
noncommutative $L^p$-spaces. We give either characterizations  
of Ritt operators satisfying the equivalent conditions
of Corollary \ref{6Equiv}, or exhibit classes of examples 
satisfying these conditions.

\bigskip\noindent
{\it 8.a. Hilbert spaces.}
Let $H$ be a Hilbert space. Two bounded operators $S,T\colon H\to H$ are called
similar provided that there is an invertible operator $V\in B(H)$ such that
$S=V^{-1}TV$. In particular we say that $T$ is similar to a contraction if 
there is an invertible operator $V\in B(H)$ such that $\norm{V^{-1}TV}\leq 1$.
This is equivalent to the existence of an equivalent Hilbertian norm
on $H$ with respect to which $T$ is contractive. Any $T$ similar to a 
contraction is polynomially bounded (by von Neumann's inequality).
Pisier's negative solution to the Halmos problem asserts that the converse
is wrong, see \cite{P4} for details and complements on similarity problems. 
It is known however that any Ritt
operator which is polynomially bounded is necessarily similar 
to a contraction, see \cite{LM2, DeL}. The next statement 
(which may be known to some similarity specialists) is a
refinement of that result, also containing the Hilbert space
version of Theorem \ref{1MainLp}.

Note that the class of Ritt operators
is stable under similarity.

\begin{theorem}\label{8Hilbert}
For any power bounded operator $T\in B(H)$, the following assertions are equivalent.
\begin{enumerate}
\item [(i)] $T$ is a Ritt operator which admits a bounded 
$H^\infty(B_\gamma)$ functional calculus for some $\gamma\in\bigl(0,\frac{\pi}{2}\bigr)$.
\item [(ii)] $T$ and $T^*$
both  satisfy uniform estimates
$$
\norm{x}_{T,1}\,\lesssim\,\norm{x}\qquad\hbox{and}\qquad
\norm{y}_{T^*,1}\,\lesssim\,\norm{y}
$$
for $x,y\in H$.
\item [(iii)] $T$ is a Ritt operator and $T$ is similar to a contraction.
\end{enumerate}
\end{theorem}

\begin{proof}
It follows from \cite[Thm. 4.7]{KP} that $T$ is a Ritt operator if it satisfies (ii).
With this result in hands, the equivalence between (i) and (ii) reduces to Corollary
\ref{6Equiv}. 

If $T$ satisfies (iii), then it is polynomially bounded (see the discussion above).
Hence it satisfies (i) by Proposition \ref{6Reduce}.

Assume (ii). Recall (\ref{5MET}) (with $X=H$) and let $P\colon H\to H$ be the 
projection onto ${\rm Ker}(I-T)$ whose kernel equals 
$\overline{{\rm Ran}(I-T)}$. Then we have an equivalence
\begin{equation}\label{8Square}
\norm{x}\,\approx\bigl(\norm{P(x)}^2\, +\,\norm{x}_{T,1}^2\bigr)^{\frac{1}{2}},
\qquad x\in H.
\end{equation}
Indeed this follows from the proof of \cite[Thm. 4.7]{KP}, see also \cite[Cor. 3.4]{LMX1}.
Let $\vert\vert\vert x \vert\vert\vert$ denote the right handside of (\ref{8Square}). 
Then $\vert\vert\vert \cdotp \vert\vert\vert$ is an equivalent Hilbertian norm on
$H$.
Further for any $x\in H$,
$$
\norm{T(x)}_{T,1}^2 = \sum_{k=1}^{\infty} k\bignorm{T^{k+1}(x)-T^{k}(x)}^2
\,\leq \sum_{k=2}^{\infty} k\bignorm{T^{k}(x)-T^{k-1}(x)}^2\,\leq \norm{x}_{T,1}^2.
$$
This implies that  $T$ is a contraction on $\bigl(H,\vert\vert\vert \cdotp \vert\vert\vert\bigr)$. 
Thus $T$ is similar to a contraction, which shows (iii). 
\end{proof}

A natural question (also making sense on general Banach spaces) is whether 
one can get rid of one of the two square function estimates of (ii)
in the above equivalence result. It turns out that the answer is negative.

\begin{proposition}\label{7One} There exists a Ritt operator $T$ on Hilbert space
which is not similar to a contraction, although it satisfies an estimate
$$
\norm{x}_{T,1}\,\lesssim\,\norm{x},\qquad x\in H.
$$
\end{proposition}

\begin{proof}
This is a simple adaptation of \cite[Thm. 5.2]{LM5} so we will be brief.
Let $H$ be a separable infinite dimensional Hilbert space and let 
$(e_m)_{m\geq 1}$ be a normalized Schauder basis on H which satisfies
an estimate
\begin{equation}\label{81}
\Bigl(\sum_m \vert t_m\vert^2\Bigr)^{\frac{1}{2}}\,\lesssim\,\Bignorm{\sum_m 
t_m e_m}
\end{equation}
for finite sequences $(t_m)_{m\geq 1}$ of complex numbers but for which
there is no reverse estimate, that is,
\begin{equation}\label{82}
\sup\biggl\{\Bignorm{\sum_m 
t_m e_m}\, :\, \sum_m \vert t_m\vert^2\leq 1\biggr\}\,=\,\infty.
\end{equation}
Let $T\colon H\to H$ be defined by letting
$$
T\Bigl(\sum_m t_m e_m\Bigr)\,=\, \sum_m (1-2^{-m})t_m e_m.
$$
According to e.g. \cite[Thm. 4.1]{LM1}, this operator is well-defined
and $A=I-T$ is sectorial of any positive type. Moreover 
$\sigma(T)\subset [0,1]$, hence $T$ is a Ritt operator.

Arguing as in the proof of \cite[Thm. 5.2]{LM5}, one obtains 
an equivalence 
$$
\Bignorm{\sum_m t_m e_m}_{T,1}\,\approx\, \Bigl(\sum_m \vert t_m\vert^2\Bigr)^{\frac{1}{2}}
$$
for finite sequences $(t_m)_{m\geq 1}$ of complex numbers. 

In view of (\ref{81}), this implies the square function estimate $\norm{x}_{T,1}\lesssim\norm{x}$. 
If $T$ were similar to a contraction, it would satisfy an  
estimate $\norm{y}_{T^*,1}\lesssim\norm{y}$, by Theorem \ref{8Hilbert}. It would therefore satisfy
a reverse estimate $\norm{x}\lesssim\norm{x}_{T,1}$ by (\ref{8Square}). 
This contradicts (\ref{82}).
\end{proof}

\bigskip\noindent
{\it 8.b. Commutative $L^p$-spaces.} Let $(\Omega,\mu)$ be a measure
space and let $1<p<\infty$. The following is the main result of
\cite{LMX1}, we provide a proof in the light of the present paper.

\begin{theorem}\label{8LMX} \cite{LMX1} Let $T\colon L^p(\Omega)\to L^p(\Omega)$
be a positive contraction and assume that $T$ is a Ritt operator. Then 
it satisfies the equivalent conditions of Theorem \ref{1MainLp}.
\end{theorem}

\begin{proof}
Let $(T_t)_{t\geq 0}$ be the uniformly continuous semigroup on $L^p(\Omega)$ defined by
$$
T_t = e^{-t} e^{tT},\qquad t\geq 0.
$$
Then for any $t\geq 0$, 
$T_t$ is positive and $\norm{T_t}\leq e^{-t} e^{t\norm{T}}\leq 1$.
The generator of $(T_t)_{t\geq 0}$ is $T-I=-A$ and since $T$ is a Ritt operator,
$A$ is sectorial of type $<\frac{\pi}{2}$. Hence $A$ admits a bounded 
$H^\infty(\Sigma_\theta)$ functional calculus for some $\theta<\frac{\pi}{2}$, by \cite[Prop. 2.2]{LMX1}.
According to Proposition \ref{4TA}, this implies condition (i) of Theorem 
\ref{1MainLp}.
\end{proof}

For applications of this result to ergodic theory, see \cite{LMX2}.

Ritt operators on $L^p(\Omega)$ satisfying Theorem \ref{1MainLp} do not have
any description comparable to the one given by Theorem \ref{8Hilbert}
on Hilbert space. However in a separate joint work with C. Arhancet \cite{ALM}, we show
that for an $R$-Ritt operator $T\colon L^p(\Omega)\to L^p(\Omega)$, 
$T$ satisfies the conditions of Theorem \ref{1MainLp} if and only 
if there exists a second measure space $(\Omega',\mu')$, two bounded
maps $J\colon L^p(\Omega)\to L^p(\Omega')$ and $Q\colon L^p(\Omega')\to L^p(\Omega)$,
as well as an isomorphism $U\colon L^p(\Omega')\to L^p(\Omega')$ such that
$\{U^n\, :\, n\in\Zdb\}$ is bounded and
$$
T^n = QU^n J,
\qquad n\geq 0.
$$

\bigskip\noindent
{\it 8.c. Noncommutative $L^p$-spaces.} In this subsection, we let $M$ be a semifinite
von Neumann algebra equipped with a semifinite faithful trace $\tau$. 
Thanks to the noncommutative Khintchine inqualities (\ref{2KI1}) and
(\ref{2KI2}), Corollary \ref{6Equiv} has a specific form on $L^p(M)$.
We state it in the case $2\leq p<\infty$, the dual case ($1<p\leq 2$) 
can be obtained by changing $T$ into $T^*$. This is the noncommutative analog 
of Theorem \ref{1MainLp}, the 
square functions (\ref{1SFLp}) 
being replaced by their natural noncommutative versions.

\begin{corollary}\label{7NcLp} Let $2\leq p<\infty$
and let $T\colon L^p(M)\to L^p(M)$ be a Ritt opertor. Then 
$T$ admits a bounded
$H^{\infty}(B_\gamma)$ functional calculus for some $\gamma<\frac{\pi}{2}\,$ if and only if
$T$ is $R$-Ritt and there exists a constant $C>0$ such that the following
three estimates hold:
\begin{enumerate}
\item [(1)] For any $x\in L^p(M)$,
$$
\biggnorm{\biggl(\sum_{k=1}^{\infty} k\bigl 
\vert T^{k}(x) - T^{k-1}(x)\bigr\vert^2\biggr)^{\frac{1}{2}}}_{L^p(M)}
\,\leq\,C \norm{x}_{L^{p}(M)}.
$$
\item [(2)] For any $x\in L^p(M)$,
$$
\biggnorm{\biggl(\sum_{k=1}^{\infty} k\bigl 
\vert \bigl(T^{k}(x) - T^{k-1}(x)\bigr)^*\bigr\vert^2\biggr)^{\frac{1}{2}}}_{L^p(M)}
\,\leq\, C\norm{x}_{L^{p}(M)}.
$$
\item [(3)] For any $y\in L^{p'}(M)$, there exists two sequences $(u_k)_{k\geq 1}$ and 
$(v_k)_{k\geq 1}$ in $L^{p'}(M)$ such that 
$$
k^{\frac{1}{2}}\bigl(T^{*k}(y) - T^{*(k-1)}(y) \bigr)
= u_k +v_k
$$
for any $k\geq 1$, and
$$
\Bignorm{\Bigl(\sum_{k=1}^{\infty} \vert 
u_k\vert^2\Bigr)^{\frac{1}{2}}}_{L^{p'}(M)}\,\leq C\norm{y}_{L^{p'}(M)}
\qquad\hbox{and}\qquad
\Bignorm{\Bigl(\sum_{k=1}^{\infty} \vert 
v_k^*\vert^2\Bigr)^{\frac{1}{2}}}_{L^{p'}(M)}\,\leq C\norm{y}_{L^{p'}(M)}.
$$
\end{enumerate}
\end{corollary}

We will now exhibit two classes of examples satisfying the conditions of the above corollary. 
We start with Schur multipliers. Here our von Neumann algebra is $B(\ell^2)$, the trace $\tau$
is the usual trace and the associated noncommutative $L^p$-spaces are the Schatten
classes that we denote by $S^p$. We represent any element of $B(\ell^2)$ by a bi-infinite
matrix is the usual way. We recall that a bounded Schur multiplier on $B(\ell^2)$
is a bounded map $T\colon B(\ell^2)\to B(\ell^2)$ of the form
\begin{equation}\label{7t}
[c_{ij}]_{i,j\geq 1}\,\stackrel{T}{\longmapsto}  \, [t_{ij}c_{ij}]_{i,j\geq 1}
\end{equation}
for some matrix $[t_{ij}]_{i,j\geq 1}$ of complex numbers. See e.g. \cite[Thm 5.1]{P4}
for a description of those maps. It is well-known (using duality and interpolation)
that any bounded Schur multiplier $T\colon B(\ell^2)\to B(\ell^2)$
extends to a bounded map $T\colon S^p\to S^p$ for any $1\leq p<\infty$, with
$$
\norm{T\colon S^p\longrightarrow S^p}\leq\norm{T\colon B(\ell^2)
\longrightarrow B(\ell^2)}.
$$
In particular, any contractive Schur multiplier $T\colon B(\ell^2)\to B(\ell^2)$
extends to a contraction on $S^p$ for any $p$. In this case, the complex numbers
$t_{ij}$ given by (\ref{7t}) have modulus $\leq 1$. Moreover, $T\colon S^2 \to S^2$
is selfadjoint (in the usual Hilbertian sense) if and only if
the associated matrix $[t_{ij}]_{i,j\geq 1}$ is real valued.

We say that a semigroup $(T_t)_{t\geq 0}$ of contractive Schur multipliers on $B(\ell^2)$
is $w^*$-continuous if $w^*$-$\lim_{t\to 0} T_t(x)=x$ for any $x\in B(\ell^2)$.
In this case,
$(T_t)_{t\geq 0}$ extends to a strongly continuous semigroup of $S^p$
for any $1\leq p<\infty$. Further we say that $(T_t)_{t\geq 0}$
is selfadjoint provided that $T_t\colon S^2 \to S^2$
is selfadjoint for any $t\geq 0$. See \cite[Chapter 5]{JLX} for 
the more general notion of noncommutative diffusion semigroup.

In the sequel we let $\omega_p=\pi\bigl\vert\frac{1}{p} -\frac{1}{2}\bigr\vert\,$.
The following extends \cite[8.C]{JLX}.

\begin{proposition}\label{7Schur1}
Let $(T_t)_{t\geq 0}$ be a selfadjoint $w^*$-continuous
semigroup of contractive Schur multipliers on $B(\ell^2)$.
For any $1<p<\infty$, let $-A_p$ be the infinitesimal 
generator of $(T_t)_{t\geq 0}$ on $S^p$. Then for any
$\theta\in (\omega_p,\pi)$, $A_p$ admits a bounded
$H^{\infty}(\Sigma_\theta)$ functional calculus.
\end{proposition}

\begin{proof}
For any $1<p<\infty$, let $(U_{t,p})_{t\in\footnotesize{\Rdb}}$ be the 
translation semigroup on the Bochner space $L^p(\Rdb; S^p)$. Then it follows from 
\cite[Cor. 4.3 \& Thm. 5.3]{A} that for any $b\in L^1(0,\infty)$,
$$
\Bignorm{\int_{0}^{\infty} b(t) T_t\, dt\  \colon S^p\longrightarrow S^p}
\leq \Bignorm{\int_{0}^{\infty} b(t) U_{t,p}\, dt\, \colon L^p(\Rdb; S^p)\longrightarrow L^p(\Rdb; S^p)}.
$$
Let $C_p$ be the negative generator of $(U_{t,p})_{t\in\footnotesize{\Rdb}}$. By \cite[Lem. 2.12]{LM1}, the
above inequality implies that for any $\theta>\frac{\pi}{2}$ and any $f\in H^{\infty}_{0}(\Sigma_\theta)$,
$$
\norm{f(A_p)}\leq \norm{f(C_p)}.
$$
Since $S^p$ is a UMD Banach space, $C_p$ has a bounded $H^{\infty}(\Sigma_\theta)$ functional
calculus for any $\theta>\frac{\pi}{2}$ (see e.g. \cite{HP}). Hence the above estimate implies
that in turn, $A_p$ has a bounded $H^{\infty}(\Sigma_\theta)$ functional
calculus for any $\theta>\frac{\pi}{2}$.

We assumed that $(T_t)_{t\geq 0}$ is selfadjoint. Hence by \cite[Prop. 5.8]{JLX}, the
above property holds true as well
for any $\theta>\omega_p$.
\end{proof}

Recall Definition \ref{2Pb} for polynomial boundedness.

\begin{corollary}\label{7Schur2}
Let $T\colon B(\ell^2)\to B(\ell^2)$ be a contractive Schur multiplier
associated with a real valued matrix
$[t_{ij}]_{i,j\geq 1}$ and let 
$1<p<\infty$.
\begin{enumerate}
\item [(1)] The induced operator $T\colon S^p\to S^p$
is polynomially bounded.
\item [(2)] If there exists $\delta>0$ such that $t_{ij}\geq -1+\delta$ for any
$i,j\geq 1$, then the induced operator $T\colon S^p\to S^p$ is a Ritt
opertor which admits a bounded $H^\infty(B_\gamma)$ functional calculus for some
$\gamma <\frac{1}{2}$. When $p\geq 2$,
it satisfies the conditions (1)-(2)-(3) of Corollary \ref{7NcLp}.
\end{enumerate}
\end{corollary}

\begin{proof}
We first prove (2). Assume that $t_{ij}\geq -1+\delta$ for any  $i,j\geq 1$.
Then the spectrum of the selfadjoint map $T\colon S^2\to S^2$ is included 
in $[-1+\delta,1]$. Applying the Spectral Theorem, this readily implies
that $T\colon S^2\to S^2$ is a Ritt operator.  According to \cite[Lem. 5.1]{LMX1},
this implies that for any $1<p<\infty$, $T\colon S^p\to S^p$ is a Ritt operator.

For any $t\geq 0$, $T_t=e^{-t}e^{tT}$ is a contractive selfadjoint Schur multiplier.
Hence for any $1<p<\infty$, $A=I-T$ has a bounded 
$H^{\infty}(\Sigma_\theta)$ functional calculus on $S^p$ for any $\theta>\omega_p$, 
by Proposition \ref{7Schur1}. Note that $\omega_p<\frac{\pi}{2}$. Thus
the result now follows from Proposition \ref{4TA} and Corollary \ref{7NcLp}.

We now prove (1). Under our assumption, the square operator 
$T^2\colon B(\ell^2)\to B(\ell^2)$
is a contractive Schur multiplier, and its associated matrix is
$[t_{ij}^2]_{i,j\geq 1}$. Hence $T^2$ satisfies part (2) of the
present corollary. Let $1<p<\infty$.
Since polynomial boundedness is implied by the existence of a bounded
$H^\infty(B_\gamma)$ functional calculus, we deduce
from (2) that there exists a constant $K_p\geq 1$ such that 
$$
\norm{\varphi(T^2)}_{B(S^p)}\,\leq K_p\,\norm{\varphi}_{\infty,\Ddb},\qquad \varphi\in\P.
$$
Any polynomial $\varphi$ admits a (necessarily unique) decomposition 
$$
\varphi(z)=\varphi_1(z^2) + z \varphi_2(z^2)
$$
and it is easy to check that 
$$
\norm{\varphi_1}_{\infty,\Ddb}\leq\norm{\varphi}_{\infty,\Ddb}
\qquad\hbox{and}\qquad
\norm{\varphi_2}_{\infty,\Ddb}\leq \norm{\varphi}_{\infty,\Ddb}.
$$
Writing $\varphi(T)=\varphi_1(T^2) + T\varphi_2(T^2)$, we deduce that
\begin{align*}
\norm{\varphi(T)}_{B(S^p)} &\leq \norm{\varphi_1(T^2)}_{B(S^p)}+ \norm{\varphi_2(T^2)}_{B(S^p)}\\
& \leq K_p
\bigl(\norm{\varphi_1}_{\infty,\Ddb}+\norm{\varphi_2}_{\infty,\Ddb}\bigr)\\
& \leq 2K_p
\norm{\varphi}_{\infty,\Ddb}.
\end{align*}
\end{proof}

We now turn to our second class of examples.
Here we assume that $\tau$ is finite and normalized, that is, $\tau(1)=1$. In this case,
$M\subset L^p(M)$ for any $1\leq p<\infty$. Following \cite{HM, R}, we
say that a linear map $T\colon M\to M$ is a Markov map if $T$ is unital, completely positive
and trace preserving. As is well-known, such a map is necessarily normal and for any
$1\leq p<\infty$, it extends to a contraction $T_p\colon L^p(M)\to L^p(M)$.
We say that $T$ is selfadjoint if its $L^2$-realization $T_2$
is selfadjoint in the usual Hilbertian sense.

Applying the techniques developed so far,
the following analog of Corollary \ref{7Schur2} is a rather 
direct consequence of some recent work of
M. Junge, \'E. Ricard and D. Shlyakhtenko.

\begin{proposition}\label{7Markov}
Let $T\colon M\to M$ be a selfadjoint Markov map.
\begin{enumerate}
\item [(1)] For any $1<p<\infty$, the operator $T_p\colon L^p(M)\to L^p(M)$
is polynomially bounded.
\item [(2)] If $-1\notin\sigma(T_2)$, then for any $1<p<\infty$,  
$T_p\colon L^p(M)\to L^p(M)$ is a Ritt
opertor which admits a bounded $H^\infty(B_\gamma)$ functional calculus for some
$\gamma <\frac{1}{2}$. When $p\geq 2$,
it satisfies the conditions (1)-(2)-(3) of Corollary \ref{7NcLp}.
\end{enumerate}
\end{proposition}

\begin{proof}
Let $A_p=I_{L^p(M)} -T_p$ for any $1<p<\infty$. Repeating the 
method applied to deduce Corrolary \ref{7Schur2} from Proposition \ref{7Schur1},
we see that it suffices to show that for any $1<p<\infty$, 
$A_p$ is sectorial and admits a bounded 
$H^\infty(\Sigma_\theta)$ functional
calculus for some $\theta<\frac{\pi}{2}$.

For that purpose, consider 
$$
T_t = e^{-t(I-T)},\qquad t\geq 0.
$$
Then $(T_t)_{t\geq 0}$ 
is a `noncommutative diffusion semigroup' in the sense of \cite[Chapter 5]{JLX}, and
for any $1<p<\infty$, $-A_p$ is the generator of its $L^p$-realization. Hence
$A_p$ is sectorial by \cite[Prop. 5.4]{JLX}.

According to \cite{JRS}, each $T_t$ is `factorizable' in the sense of
\cite[Def. 6.2]{AD} or \cite[Def. 1.3]{HM}. Writing $T_t=T_{\frac{t}{2}}^{2}$ and 
using \cite[Thm. 5.3]{HM} we deduce that each $T_t$ satisfies the `Rota dilation 
property' introduced in \cite[Def. 10.2]{JLX} (see also \cite[Def. 5.1]{HM}).

We deduce the result by applying the reasoning in \cite[10.D]{JLX}.
Indeed  it is implicitly
shown there that whenever $(T_t)_{t\geq 0}$
is a diffusion semigroup on a finite von Neumann algebra such that each $T_t$
satisfies the Rota dilation 
property, then the negative generator of its $L^p$-realization admits a 
bounded $H^\infty(\Sigma_\theta)$ functional
calculus for any $\theta>\omega_p$.
\end{proof}

\begin{remark}\label{7Rk1}\

(1) About the necessity of having two parts in Proposition
\ref{7Markov}, we note that $L^p$-realizations of selfadjoint
Markov maps are not necessarily 
Ritt operators. For instance,
the mapping $T\colon\ell^{\infty}_{2}\to \ell^{\infty}_{2}$ defined by $T(t,s)=(s,t)$
is a Markov map but $-1\in\sigma(T)$.

\smallskip 
(2) If $T\colon M\to M$ satisfies the Rota dilation 
property, then it is a Markov map and its $L^2$-realization
is positive in the Hilbertian sense. Hence it satisfies Proposition
\ref{7Markov}. In this case, the latter statement strengthens 
\cite[Cor. 10.9]{JLX}, where weaker square
function estimates were established for opertors 
with the Rota dilation 
property. 

\smallskip 
(3) For any selfadjoint Schur multiplier (resp. Markov map) $T$, the square operator
$T^2$ satisfies the second part of Corollary \ref{7Schur2} (resp. Proposition \ref{7Markov}). 
Hence it satisfies an estimate
$$
\biggnorm{\sum_{k=1}^{\infty} k^{\frac{1}{2}}\varepsilon_k\otimes \bigl(
T^{k-1}(x) -T^{k+1}(x)\bigr)}_{{\rm Rad}(L^p(M))}\,\lesssim\, \norm{x}_{L^p}.
$$
\end{remark}

\end{document}